\documentclass[a4paper,dvipsnames]{siamart250211}
\usepackage[scale=0.8]{geometry}
\usepackage[numbers]{natbib}
\usepackage[utf8]{inputenc}

\usepackage{bbold}
\usepackage{graphicx}
\usepackage{amsmath,amssymb}
\usepackage{mathtools,mathdots,mathrsfs}

\usepackage{kbordermatrix}
\usepackage{bm}

\usepackage{setspace}

\usepackage{tikz}
\usepackage{cleveref}
\usepackage{hyperref}

 \newsiamremark{remark}{Remark}
 \newsiamthm{conjecture}{Conjecture}
 \newsiamremark{example}{Example}

\usepackage[textsize=footnotesize]{todonotes}

\newcommand{\RR}{\mathbb{R}}
\newcommand{\HH}{\mathbb{H}}
\newcommand{\PP}{\mathbb{P}}
\newcommand{\ZZp}{\mathbb{Z}_{+}}
\newcommand{\norm}[1]{\ensuremath{\left\| #1\right\|}}

\newcommand{\calK}{\mathcal{K}}

\newcommand{\vect}[1]{\boldsymbol{#1}}
\newcommand{\matr}[1]{\boldsymbol{#1}}
\newcommand{\eqdef}{\stackrel{\textrm{def}}{=}}
\newcommand{\T}{{\sf T}}        

\newcommand{\bal}{\bm{\alpha}}
\newcommand{\bbet}{\bm{\beta}}

\newcommand{\ba}{\vect{a}}

\newcommand{\bu}{\vect{u}}
\newcommand{\bv}{\vect{v}}
\newcommand{\bx}{\vect{x}}
\newcommand{\by}{\vect{y}}

\newcommand{\ep}[1]{\varepsilon^{#1}}

\newcommand{\bKe}{\matr{K}(\varepsilon)}
\newcommand{\bKpe}{\bK'(\varepsilon)}

\newcommand{\bC}{\matr{C}}

\newcommand{\bM}{\matr{M}}

\newcommand{\bA}{\matr{A}}
\newcommand{\bH}{\matr{H}}       
\newcommand{\bHt}{\tilde{\matr{H}}}       
       
\newcommand{\bHh}{\hat{\matr{H}}}

\newcommand{\bq}{\vect{q}}       

\newcommand{\bQ}{\matr{Q}}

\newcommand{\bB}{\matr{B}}       
\newcommand{\bZ}{\matr{Z}}       
\newcommand{\bU}{\matr{U}}

\newcommand{\bI}{\matr{I}}       
\newcommand{\bV}{\matr{V}}

\newcommand{\bW}{\matr{W}}       
\newcommand{\bK}{\matr{K}}       
\newcommand{\bR}{\matr{R}}       
\newcommand{\bD}{\matr{D}}       
\newcommand{\bDe}{\matr{\Delta}(\varepsilon)}       
\newcommand{\bDet}{\matr{\tilde{\Delta}}(\varepsilon)}       
\newcommand{\beps}{\bm{\varepsilon}}

\newcommand{\bS}{\matr{S}}

\newcommand{\bLa}{\matr{\Lambda}}
\newcommand{\bla}{\vect{\lambda}}

\newcommand{\R}{\mathbb{R}}       

\newcommand{\X}{\mathcal{X}}

\newcommand{\flatlim}{\varepsilon\rightarrow0}

\renewcommand{\O}{\mathcal{O}}       
\DeclareMathOperator{\rank}{rank}
\DeclareMathOperator{\mspan}{span}

\DeclareMathOperator{\argmax}{argmax}

\DeclareMathOperator{\diag}{diag}

\DeclareMathOperator{\leadt}{lt}
\DeclareMathOperator{\leadc}{lc}
\DeclareMathOperator{\val}{val}

\definecolor{darkgreen}{rgb}{0,0.6,0}

\newcommand{\he}[1]{\color{blue} \mathit{#1} \color{black}}

\newcommand{\ones}{\vect{\mathbb{1}}}

\newcommand{\zeroes}{\matr{0}}

\newcommand{\lt}{\tilde{\lambda}}

\newcommand{\Kb}{\,\underline{\!\matr{K}\!}\,}
\newcommand{\Kbe}{\underline{\matr{K}(\epsilon{})}}

\newcommand{\ASE}[1]{\underline{#1}}
\newcommand{\truncASE}[2]{\trunc{#2} \ASE{#1}}
\newcommand{\equivq}{\overset{q}{\sim}}
\newcommand{\equivpar}[1]{\overset{#1}{\sim}}
\newcommand{\equivinf}{\sim}

\newcommand{\trunc}[1]{\mathrm{trunc}_{#1}}

\newcommand{\vM}{\matr{\Omega}}
\newcommand{\vMt}{\tilde{\matr{\Omega}}}

\newcommand{\Ulim}{\tilde{\matr{U}}}

\newcommand{\rInvK}[2]{\bM({#2}\varepsilon^{#1},\bK)}

\newcommand{\TheTitle}{Computing asymptotic eigenvectors and eigenvalues of \\
perturbed symmetric matrices}
\newcommand{\TheShortTitle}{Asymptotic eigenvectors and eigenvalues of perturbed matrices}

\newcommand{\TheAuthors}{K. Usevich and S.Barthelm\'{e}}
\headers{\TheShortTitle}{\TheAuthors}

\title{\TheTitle\thanks{Submitted to the editors DATE.
\funding{This work was supported by the ANR projects MIAI@Grenoble Alpes
  (ANR-19-P3IA-0003) and LeaFleT (ANR-19-CE23-0021-01).}}}

\author{
  Konstantin Usevich\thanks{Universit\'{e} de Lorraine and CNRS, CRAN (Centre de
    Recherche en Automatique en Nancy), UMR 7039, Campus Sciences, BP 70239,
    54506 Vand\oe{}uvre-l\`{e}s-Nancy cedex, France
    (\email{konstantin.usevich@cnrs.fr}).}%
  \and
  Simon Barthelm\'{e}\thanks{CNRS, Univ. Grenoble Alpes,  Grenoble INP, GIPSA-lab, 38000 Grenoble, France (\email{simon.barthelme@gipsa-lab.fr}).}
}

\begin{document}
\maketitle
\begin{abstract}
  Computing the eigenvectors and eigenvalues of a perturbed matrix can be
  remarkably difficult when the unperturbed matrix has repeated eigenvalues. In
  this work we show how the limiting eigenvectors and eigenvalues of a symmetric
  matrix $\bK(\varepsilon)$ as $\flatlim$ can be obtained relatively easily from
  successive Schur complements, provided that the entries scale in different
  orders of $\varepsilon$. If the matrix does not directly exhibit this
  structure, we show that putting the matrix into a ``generalised kernel form''
  can be very informative. The resulting formulas are much simpler than
  classical expressions obtained from complex integrals involving the resolvent.
  
  We apply our results to the problem of computing the eigenvalues and
  eigenvectors of kernel matrices in the ``flat limit'', a problem
  that appears in many applications in statistics and approximation theory. In particular, we prove a
  conjecture from [SIAM J. Matrix Anal. Appl., 2021, 42(1):17--57]  which connects the
  eigenvectors of kernel matrices to multivariate orthogonal polynomials.  
\end{abstract}

\begin{keywords}
matrix perturbations, kernel matrices, eigenvectors, eigenvalues, flat limit, radial basis functions, tropical algebra
\end{keywords}

\begin{MSCcodes}
15A18,15B57,15A80,47A55,47A75,47B34,65F15
\end{MSCcodes}

\section{Introduction}
\label{sec:intro}
In many applications, such as scattered data approximation, machine learning and
statistics, or numerical methods for Partial Differential Equations, matrices
$\bKe$ arise that depend on a scaling parameter $\varepsilon$. Of particular
interest is the behavior of eigenvalues and eigenvectors for small
$\varepsilon$, when the matrices of interest become often very ill-conditioned.
Examples include stiffness matrices in certain PDEs \cite{kannan2014detecting},
transition matrices for nearly-reducible Markov chains \cite{sharpe2021nearly},
and many others.

The motivating example for this paper are kernel matrices, a class of matrices
with applications throughout statistics and Machine Learning
\cite{schaback2006kernel}. A common class of kernel matrices is defined by
radial basis function (RBF) kernels. Given a set of points $\bx_1,\ldots,\bx_n
\in \RR^d$ and an RBF $\psi(\cdot)$, an $n\times n$ matrix is constructed as
follows
\begin{equation}
  \label{eq:kernel-example-rbf}
  \bKe = \left[ \psi\left(\varepsilon\norm{\bx_i-\bx_j}\right) \right]_{i,j=1}^n;
\end{equation}
where  $\varepsilon$ plays the role of an inverse scale parameter.
A simplest example of Gaussian (squared exponential) RBF for $d=1$ is shown in \Cref{fig:flat-limit}.
For  small $\varepsilon$, these matrices become very ill-conditioned, as  $\bK_0= \bK(0) = \ones \ones^\T$ is a rank-one matrix 
and the eigenvalues decay very  quickly as $\varepsilon \to 0$.
Such a limit, called ``flat limit'', has attracted much attention in the approximation theory literature \cite{driscoll2002interpolation},
as small values of $\varepsilon$ are important from a practical standpoint.


In this paper we consider symmetric matrices that depend analytically on $\varepsilon$:
\begin{equation}
  \label{eq:analytic-matrix-pert}
  \bKe = \bK_0+\varepsilon \bK_1+\ep2\bK_2+\ldots, 
\end{equation}
i.e., we are dealing with analytic perturbations of the matrix $\bK_0$.
It is well-known from analytic perturbation theory that for the symmetric case, the eigenvalues and eigenvectors can be chosen to be analytic, in particular, $\lambda_k(\varepsilon) = \varepsilon^{\alpha_k}(\widetilde{\lambda}_k + \O(\varepsilon))$ and $\bu_k(\varepsilon) = \widetilde{\bu}_k + \O(\varepsilon)$.
(See \Cref{fig:flat-limit} for an illustration where $\alpha_k = 2(k-1)$, i.e., the eigenvalues are of order $\O(1), \O(\varepsilon^2), \O(\varepsilon^4),\ldots$)
The focus of this paper is on determining the leading terms of eigenvalues and eigenvectors of matrices of form \eqref{eq:analytic-matrix-pert}.

If $\bK_0$ has full rank, and simple eigenvalues, then the question is easy to
answer using regular perturbation theory \cite{CourantHilbert} (the limiting eigenvectors/eigenvalues are those of $\bK_0$). 
However, if $\bK_0$ is rank-deficient, there is more much more work involved if
we use classical perturbation theory \cite{kato1995perturbation}: we need to  perform so-called reduction of perturbation series with respect to multiple eigenvalue of $\bK_0$ at $0$.
Such an approach  becomes even more impractical if there are many different
orders  $\alpha_1 \le \cdots \le \alpha_n$ of  eigenvalues, as in that case
heavy recursive reduction of perturbation series would be required.

For symmetric positive definite kernel matrices, however, much more is known. In
this case, determinantal identities can be applied to retrieve eigenvalues
(without resorting to perturbation expansions) and the Courant-Fischer principle
can be used in certain cases to find the limiting eigenvectors. This leads to a
complete characterization of limiting eigenvalues and eigenvectors for the case
$d=1$ (points on a real line), see
\cite{schaback2005multivariate,BarthelmeUsevich:KernelsFlatLimit} for most known
kernel matrices. This covers the example in \Cref{fig:flat-limit}, where the
limits and the limiting behaviour can be computed analytically. For $d>1$
(higher-dimensional problems), the situation is more difficult because there are
groups of eigenvalues of different orders in $\varepsilon$ (see the first
example in \cref{sec:numerics}). For such a case, the most recent results, due
to \cite{wathen2015eigenvalues}, and \cite{BarthelmeUsevich:KernelsFlatLimit},
establish limiting eigenvalues and group projectors. However, it is not possible
to obtain the behavior of individual eigenvectors in groups of different orders
using these techniques; in addition, for eigenvalues, the determinantal
identities and proof techniques become much more involved (see
\cite{BarthelmeUsevich:KernelsFlatLimit}).

\begin{figure}
\center \includegraphics[width=14cm]{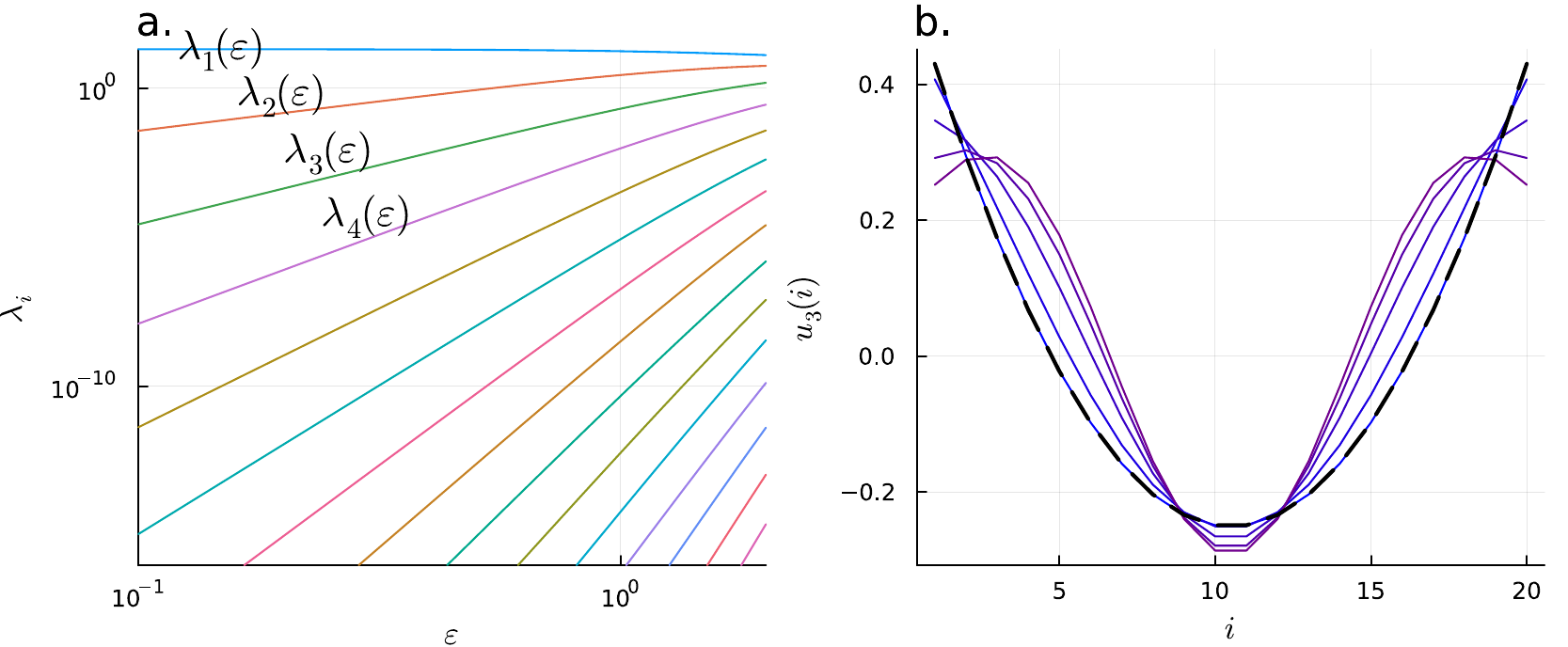}
  \caption{\label{fig:flat-limit} Eigenvalues and eigenvectors of a kernel matrix in small $\varepsilon$. We take 20 equispaced points $x_1, \ldots, x_n$ in $[0,1]$ and form the matrix defined in eq. \ref{eq:kernel-example-rbf}, for $d=1$ and $\psi(t) = \exp(-t^2)$. \textbf{a.}  The eigenvalues of $\bKe$ (computed numerically) as a function of $\varepsilon$, on log-log axes. All eigenvalues except the largest one go to 0 as $\flatlim$, and they do so at increasing rate. \textbf{b.} The eigenvector $u_3(i)$ corresponding to $\lambda_3(\varepsilon)$, plotted as a function of $i$. We show the eigenvector for different values of $\varepsilon$, going from blue to purple as $\flatlim$. We see that the eigenvector converges to a well defined limit (shown as the dotted black line) even though the associated eigenvalue goes to 0. Our goal in this paper is to analyse the asymptotic spectral behaviour of matrices that depend analytically on a parameter $\varepsilon$, and are singular at $\varepsilon=0$. \Cref{thm:ase-smooth} is applicable to the kernel matrix shown here, and provides an expansion for the eigenvalues, as well as an expression for the limiting eigenvectors. }
\end{figure}




In this work we provide a set of tools for determining the limiting eigenvectors
and eigenvalues of singular symmetric (not necessarily positive definite) analytic perturbations
\eqref{eq:analytic-matrix-pert}. We show that:
\begin{enumerate}
\item For a large class of matrices, the information about the limiting
  eigenvalues and eigenvectors can be easily extracted from the Schur
  complements of a certain block matrix. This approach covers matrices in what
  we call the ``generalized kernel form'' (GKF), i.e., matrices of the form
  $\bKe =\bV \bDe (\bH + \O(\varepsilon)) \bDe \bV^{\T}$, where $\bDe$ is a
  diagonal matrix. In this case, the Schur complements of the coefficient matrix
  $\bH$ determine the limiting spectral properties. Our main result relies on
  regularised inverses, inspired by an technique used in
  \cite{barthelme2021determinantal} for the analysis of Determinantal Point Processes.
\item Kernel matrices are a special case of GKF (both for finitely smooth and
  infinitely smooth kernels), and using that we provide as a corollary the
  complete characterization of limiting eigenvectors and eigenvalues. The
  results highlight the relation of eigenvectors to multivariate discrete
  orthogonal polynomials. This settles the conjecture formulated in
  \cite{BarthelmeUsevich:KernelsFlatLimit} for $d>1$ and supersedes the results
  of \cite{wathen2015eigenvalues}, as the result we obtain apply also to
  non-unisolvent sets of points, and are thus applicable to
  approximation and interpolation on curves and surfaces. In addition, all the
  results of \cite{BarthelmeUsevich:KernelsFlatLimit} can be obtained in a much
  more straighforward way, without using determinantal identities or
  the perturbation series of Kato.
\item In some cases, the generalised kernel form may not be sufficient to
  characterise all limiting eigenvectors and eigenvalues. In this case, we
  propose an iterative algorithm, again based on Schur complements, to extract
  the relevant information at increasing orders of $\varepsilon$.
\end{enumerate}

\paragraph{Relation to previous works}
Perturbation theory for matrices goes back to the work of Rayleigh \&
Schrödinger in physics, where the setting is often informal (see
\cite{CourantHilbert}). In the mathematical literature, there are two different
strands, as identified in the recent review by Greenbaum et al.
\cite{greenbaum2020first}. One is concerned with proving perturbation bounds
(see for instance \cite{Bhatia1997,stewart1990matrix}), and the other is
analytic perturbation theory, with the key references being the books by Kato
\cite{kato1995perturbation} and Baumgärtel \cite{baumgaertel1984analytic}. Our
results have a close kinship with the approaches in the tropical algebra
literature \cite{akian2004min,akian2016non,hook2015max} (in particular, the
diagonal scalings we use are strongly related to the so-called ``Hungarian
scalings'' used in tropical algebra \cite{hook2019max}). Compared to classical
approaches in analytic perturbation theory we do not use complex integrals at
all but rely on properties of matrix functions \cite{Higham2008}. Our results
are expressed in terms of Schur complements, which show up in other works as
well (see \cite{carlsson2018perturbation,Carlsson2024}).

The paper is structured as follows. In \cref{sec:gener-matr-pert} we recall the
main properties of analytic matrix perturbations  and set up notation that we
will use in what follows. In particular, we introduce the key notion of
\emph{Asymptotic Spectral Equivalent} (ASE), that compactly encodes the
information on limiting eigenvalues and eigenvectors. We show how the Asymptotic
Spectral Equivalent is linked to regularised inverses, by using standard
properties of matrix functions \cite{Higham2008}. The main results on limiting
eigenvalues and eigenvectors are contained in \cref{sec:diagonal-scalings} for
diagonally-scaled matrices ($\bKe =\bDe (\bH + \O(\varepsilon)) \bDe$) and in
\cref{sec:generalised-kernel-matrices} for matrices in the generalised kernel
form. These results rely upon the results on the ASE and regularised inverses in
\cref{sec:gener-matr-pert}, and do not use the tropical formalism of
\cite{akian2004min}.  \Cref{sec:kernel-matrices} contains application of the
results to the case of kernel matrices and shows how to treat in a unified way
smooth and finitely smooth kernels, in the unisolvent as well as in the
non-unisolvent case (thus proving and generalising results and conjectures from
\cite{barthelme2021determinantal}). Finally, we discuss in
\cref{sec:iterative-alg} what can be done in the case when the generalised
kernel form is not sufficiently informative.

\section{Analytic eigenvalue decompositions and related tools}
\label{sec:gener-matr-pert}
Matrix perturbation theory is an old and large field, and often surprisingly
intricate. Studying general perturbations of general matrices (or worse,
operators) can be very difficult. We focus on \emph{symmetric} matrices, which
are more tractable. We study analytic matrix perturbations, of the form \eqref{eq:analytic-matrix-pert}.

Our goal is to characterise the spectral behaviour of $\bKe$ as $\flatlim$,
i.e., what are the eigenvalues and eigenvectors like for small $\varepsilon$? If
$\bK_0$ has full rank, and simple eigenvalues, then the question is easy to
answer using regular perturbation theory \cite{CourantHilbert}. If $\bK_0$ has
repeated eigenvalues, for instance when it is rank-deficient (as for kernel matrices), a lot more work is involved.
Fortunately, we have the following result, due to Rellich, described in the next subsection.

\subsection{Analytic eigenvalue decomposition of symmetric matrices}
\begin{theorem}
  \label{lem:rellichs-theorem}
  Let $\bK(\varepsilon)$  be as in \eqref{eq:analytic-matrix-pert}, with
  $\bK(\varepsilon)$ symmetric. The eigenvalues $\lambda_1(\varepsilon), \ldots, \lambda_n(\varepsilon)$ and the
  corresponding normalized eigenvectors $\bu_1(\varepsilon), \ldots,
  \bu_n(\varepsilon)$ (i.e., satisfying $\|\bu_{k} (\varepsilon)\|_2 = 1$)  may
  be chosen analytic in a (complex) neighbourhood of $0$.
\end{theorem}
\begin{remark}
In matrix notation, the analytic eigenvalue decomposition can be written as 
\[
\bKe = \bU(\varepsilon) \bLa(\varepsilon) \bU^{\T}(\varepsilon),
\]
where 
\[
\bLa  =  
\begin{pmatrix}\lambda_1(\varepsilon)  &&  \\& \ddots & \\ & & \lambda_n(\varepsilon)\end{pmatrix}, \quad 
\bU(\varepsilon) = \begin{pmatrix}\bu_1(\varepsilon) &  \cdots & \bu_n(\varepsilon)\end{pmatrix}.
\]
The orthogonality and normalization of eigenvectors imply that the eigenvector matrix satisfies the constraints
\begin{equation}\label{eq:normalization_U}
\bU^{\T}(\varepsilon) \bU(\varepsilon) =  \bU(\varepsilon)  \bU^{\T}(\varepsilon) = \bI.
\end{equation}
\end{remark}
\begin{remark}
Note \eqref{eq:normalization_U} requires that the eigenvectors are of norm $1$ for all $\varepsilon$ under consideration.
This constraint can be relaxed to require that   $\bU^{\T}(\varepsilon) \bU(\varepsilon)$ is a diagonal matrix (of the form $\matr{I} + \O(\varepsilon)$), see  \cite{greenbaum2020first} for a related discussion.
\end{remark}

In our work we will be concerned with finding the limiting eigenvalues and eigenvectors, that is finding $\bU(0)$, as well as the leading terms in the expansion of $\lambda(\varepsilon)$ (see \cref{sec:ASE} for a precise definition).
For small matrices (up to $n \le 4$), the eigenvalues can be computed by finding the roots of characteristic polynomial.
In the general case, $n > 4$, 
classical analytic perturbation theory (see e.g. the book by Kato \cite{kato1995perturbation}) provides an exhaustive construction of perturbation series for $\matr{U}(\varepsilon)$ and $\matr{\Lambda}(\varepsilon)$, by using the tools of complex analysis and an expansion of the the \emph{resolvent} $\bR_\varepsilon(z) =
(z \bI-\bKe)^{-1}$. However, the resulting perturbation series are often complicated and difficult to work with.
Moreover,  we are dealing with rank-deficient $\bK(0)$, and all the eigenvalues may have different leading exponents in $\varepsilon$, as shown in the following example.
The approach \cite{kato1995perturbation} is not well adapted to our case, as it typically proceeds by recursion over groups of eigenvalues.

In order to find the leading terms, we follow a different approach (related to \cite{akian2004min,lidskii1966perturbation}), as we need only the leading terms in the expansions.
The approach consists in bringing the matrix into so-called diagonally scaled form, and then uses regularized inverses and Schur complements, and is described in detail in \cref{sec:diagonal-scalings}.
To give a preview of the results, we consider a specific $5
\times 5$ matrix, which serves as a running example for the paper.

\begin{example}
  \label{ex:5x5}
    Consider
  \[
    \bKe =
    \begin{pmatrix}
      1 & \frac{\varepsilon}{2} & \varepsilon^4 & 0 & 0 \\
      \frac{\varepsilon}{2} & \frac{1}{4} \varepsilon^2 & \frac{\varepsilon^2}{2} & 0 & 0\\
      \varepsilon^4 & \frac{\varepsilon^2}{2} & \varepsilon^2 & \frac{\varepsilon^3}{2} & 0 \\
       0  & 0 &  \frac{\varepsilon^3}{2} & \frac{1}{8} \varepsilon^4 & \frac{\varepsilon^4}{2} \\
        0 & 0 & 0 & \frac{\varepsilon^4}{2} & \varepsilon^4 
    \end{pmatrix}.
  \]
  To compute the eigenvalues of such a matrix, the naïve approach which consists
  in finding the roots of the characteristic polynomial does not work,
  since there is no closed-form formula for the roots of a degree 5 polynomial. 
  The tools  described in  \cref{sec:diagonal-scalings} are applicable
  however, and tell us that the eigenvalues have expansion:
  \begin{align*}
    \lambda_1(\varepsilon) &= \lt_1+\O(\varepsilon), \\
    \lambda_2(\varepsilon) &= \varepsilon^2( \lt_2 + \O(\varepsilon) ),\quad \lambda_3(\varepsilon) = \varepsilon^2( \lt_3 + \O(\varepsilon) ), \\
    \lambda_4(\varepsilon) &= \varepsilon^4( \lt_4 + \O(\varepsilon) ), \quad  \lambda_5(\varepsilon) = \varepsilon^4( \lt_5 + \O(\varepsilon) ) 
  \end{align*}
  where
  $\lt_1=1,\lt_2=\frac{1+\sqrt{2}}{2},\lt_3=\frac{1-\sqrt{2}}{2},\lt_4=\frac{9+\sqrt{113}}{16},
  \lt_5 = \frac{9-\sqrt{113}}{16} $. There are thus three groups of eigenvalues: one eigenvalue that
  does not go to $0$ (leading exponent $\varepsilon^0$), $2$ eigenvalues that go to $0$ at rate
  $\varepsilon^2$ (leading exponent $\varepsilon^2$), and
  two other eigenvalues that go to $0$ at rate $\varepsilon^4$ (leading exponent $\varepsilon^4$).
  The matrix $\Ulim$ of asymptotic eigenvectors is given by:
  \[ 
  \bU(\varepsilon) =
    \left( 
    \begin{array}{c|cc|cc}
      1 & 0 & 0 &0 & 0\\
      0&    0.38 & -0.92 &0 &0 \\
      0 & 0.92 & 0.38 & 0&0 \\
      0  & 0& 0 & 0.41 & -0.91\\
      0& 0&0 & -0.91 & 0.41
    \end{array}
  \right) + \O(\varepsilon),
\]
We report numerical values (up to two digits), exact expressions are available
but lengthy. The vertical bars separate the three groups of eigenvectors. The
calculations are explained  later in \cref{sec:diagonal-scalings}. 
\end{example}
\subsection{Notation and assumptions}
\label{sec:notation}

All matrix perturbations $\bKe$ considered here are symmetric, real and
analytic: $\forall \varepsilon \in \R, \bKe \in \R^{n \times n}, \bKe^\T = \bKe$. We do
\emph{not} assume that $\bKe$ is positive definite. Our results can be extended to linear operators in Hilbert spaces by treating $\bKe$ as a 
$\infty \times \infty$ pseudo-matrix, but we take $n$ finite for simplicity. 

We need some notation related to power series.
\begin{definition}
  Let $p(\varepsilon)= \sum_{i=0}^\infty a_i\varepsilon^i $ a power series in
  $\varepsilon$. Then:
  \begin{itemize}
  \item The \emph{leading term}, noted $\leadt(p)$ is the first non-zero term
  \item The \emph{leading coefficient}, noted $\leadc(p)$ is the coefficient of $\leadt(p)$
  \item The \emph{valuation}, noted $\val(p)$ is the degree of $\leadt(p)$
  \item The \emph{leading monomial} is $\varepsilon^{\val(p)}$
  \item The \emph{truncation} of $p$ to degree $k$ is the series
    $\trunc{\varepsilon^k}(p) = \sum_{i=0}^k a_i \varepsilon^i$.
  \end{itemize}
\end{definition}
\begin{example}
  Let $p(\varepsilon) = 2\ep3 + 3\varepsilon^5 + \varepsilon^7$. Then $\leadt(p) =
  2\ep3$, $\leadc(p) = 2$, $\val{p} = 3$, the leading monomial is $\ep3$ and $\trunc{\varepsilon^5}(p)=2\ep3+3\varepsilon^5$.
\end{example}

Some of the facts from \cite{kato1995perturbation} are essential and will let us
set up notation and assumptions. We assume throughout that $\bK(\varepsilon)$ has
dimension $n \times n$ and is symmetric. Its
limit $\bK(0)$ is rank-deficient with rank $c_0 < n$. In such a case results from
\cite{kato1995perturbation} tell us that the eigenvalues of $\bK(\varepsilon)$
have the following behaviour in small $\varepsilon$:
\begin{itemize}
\item $c_0$ eigenvalues have valuation $0$ in $\varepsilon$, i.e. an expansion of
  the form $\lambda(\varepsilon) = \lt + \O(\varepsilon)$, with $\lt\neq0$. These go
  to the non-zero eigenvalues of $\bK(0)$ in the limit (i.e. $\lt$ is a non-zero
  eigenvalue of $\bK_0$) 
\item The other eigenvalues come in groups with increasing valuation; depending
  on the other terms of $\bKe$ as a power series, there may be a group with
  valuation 1, a group with valuation 2, etc. 
\end{itemize}

We group eigenvalues asymptotically by valuation. We note the valuations
$\alpha_0,\alpha_1,\ldots,\alpha_p$, so that there are $p+1$ groups of
eigenvalues (generally, $\alpha_0=0$). The valuations are increasing:
$\alpha_i\leq \alpha_{i+1}$ . The number of eigenvalues in group $i$ is denoted
$c_i$. The eigenvalues in group $i$ ($\lambda_{i,k}(\varepsilon), k = 1,\ldots, c_i$)  have expansion
\begin{equation}\label{eq:eig_expansion_individual}
\lambda_{i,k} = \varepsilon^{\alpha_i}\left(\lt_{i,k}+ O(\varepsilon)\right),
\end{equation}
where  within each group we order eigenvalues in
decreasing $\lt_{i,k}$, so that for small enough $\varepsilon$,
$\lambda_{i,k}(\varepsilon) \geq \lambda_{i,k+1}(\varepsilon)$. Note that some of
these eigenvalues can be negative. 

The eigenvectors expand as $\bU(\varepsilon) = \Ulim + \varepsilon\bU^{(1)} +
\dots$. Here we are only interested in computing $\Ulim$, which we partition
as
\begin{equation}
  \label{eq:U0}
  \Ulim =
  \begin{pmatrix}
    \Ulim_0 & \Ulim_1 & \ldots & \Ulim_p
  \end{pmatrix}
\end{equation}
according to the eigenvalues they are associated with. 
$\Ulim_i \in \R^{n \times
  c_i}$ contains the $c_i$ limiting eigenvectors associated with the $i$-th
group of eigenvalues.
We also use the following compact notation for the expansion of the $i$-th group of eigenvalues 
\eqref{eq:eig_expansion_individual} and their leading terms 
\[ \bla_i(\varepsilon) = \varepsilon^{\alpha_i}\left(\vect{\lt}_i + O(\varepsilon)
  \right).\]

\begin{example}\label{ex:5x5-notation}
In the $5\times5$ matrix of  \Cref{ex:5x5}, we get 
\[
(\alpha_0,c_0) = (0,1), (\alpha_1,c_1)  = (2,2),  (\alpha_2,c_2) = (4,2).
\]
meaning that there is 1 eigenvalue with valuation 0, 2 eigenvalues with valuation 2, and 2 eigenvalues of valuation 4.

The notation for the block of eigenvalues becomes
\[
\vect{\lt}_0 = 1, \vect{\lt}_1 =\begin{pmatrix} 2.08 \\ 0.57 \end{pmatrix} ,
\vect{\lt}_2 =  \begin{pmatrix}2.18\\0.77\end{pmatrix},
\]
(truncated to to $2$ digits of accuracy)
and the eigenvector blocks, respectively
\[
\Ulim_0 =  \begin{pmatrix} 1 \\0\\0\\0\\0\end{pmatrix}, \quad
\Ulim_1 = \begin{pmatrix}0&0\\0.38 & -0.92\\ 0.92 & 0.38\\ 0&0\\0&0\end{pmatrix}, \quad
\Ulim_2 = \begin{pmatrix}0&0\\0&0\\0&0\\0.41 & -0.91\\ -0.91 & 0.41\end{pmatrix}.
\]
\end{example}

\section{The Asymptotic Spectral Equivalent}
\label{sec:ASE}

In the literature on analytic perturbation theory, results are often given in
terms of series expansions for a particular eigenvalue, or groups of
eigenvalues, and the associated eigenvectors or eigenprojectors. In this section
we introduce an operator we call the \emph{Asymptotic Spectral Equivalent}
(ASE), which lets us state our results on limiting spectral behaviour in a
compact and unified way. As we will explain, the ASE of a matrix $\bKe$ is a
\emph{simpler} matrix, whose limiting spectral behaviour is obvious from
inspection, and which matches the limiting spectral behaviour of $\bKe$.

We begin by defining the operator and a few of its properties. We then use it to
define a notion of asymptotic equivalence, which characterises the class of
matrices with equivalent limiting spectral behaviour up to a certain order.
Finally, we prove a characterisation of asymptotic equivalence using limiting
regularised inverses, which is the central tool for the proofs in \cref{sec:reg-inverse}.

\subsection{Definition and basic properties}
\label{sec:ase-def-basic-properties}

Given an analytic matrix perturbation $\bKe$, we can form another matrix, noted
$\Kbe$, which we call the ``Asymptotic Spectral Equivalent''. $\Kbe$ is also a
matrix perturbation, which shares the asymptotic spectral properties of $\bKe$,
but whose particular form makes those properties easy to read out.

\begin{definition}[Asymptotic Spectral Equivalent]
  \label{def:ASE}
  Let $\bKe = \bU(\varepsilon)\bLa(\varepsilon)\bU(\varepsilon)^\T$ a (symmetric)
  analytic matrix perturbation. We define the Asymptotic Spectral Equivalent
  of $\bKe$ as
  \begin{equation}
    \label{eq:ASE}
    \Kbe \eqdef \Ulim \leadt(\bLa) \Ulim = \sum_{i=0}^p \varepsilon^{\alpha_i}\Ulim_i \diag(\vect{\lt}_i)\Ulim_i^\T. 
  \end{equation}
\end{definition}
We often write
\begin{equation}
  \label{eq:ASE-terms}
  \Kbe = \sum_{i=0}^p \varepsilon^{\alpha_i} \Kb_i
\end{equation}
Theorems given below provide ways to identify the terms $\Kb_0, \Kb_1,\ldots$ in \eqref{eq:ASE-terms}.
\begin{remark}[On uniqueness of ASE]
Note that while the analytic eigenvalue decomposition may be nonunique (due to sign ambiguity of eigenvectors, or  due to repeating eigenvalues $\lambda_i(\varepsilon)$), the ASE is defined uniquely.
Let $\bu_{i,k}(\varepsilon) = \widetilde{\bu}_{i,k} +\O(\varepsilon)$. 
Then the sign ambiguity is lifted by the fact that  $\Kbe$ contains the terms $\varepsilon^{\alpha_i}\lt_{i,k} \widetilde{\bu}_{i,k} \widetilde{\bu}_{i,k}^{\T}$, which do not change if we replace $\widetilde{\bu}_{i,k}$ by $-\widetilde{\bu}_{i,k}$.
Second, without loss of generality, let $\lt_{i,1}(\varepsilon) = \lt_{i,\ell}(\varepsilon)$ be a group of equal eigenvalues,
so that the eigenvectors $\widetilde{\bu}_{i,k} (\varepsilon)$ are not defined uniquely for $k=1,\ldots,\ell$.
Then we have that the corresponding term of the ASE
\[
\sum\limits_{k=0}^\ell \varepsilon^{\alpha_i}\lt_{i,k} \widetilde{\bu}_{i,k} \widetilde{u}_{i,k}^{\T} = \varepsilon^{\alpha_i}\lt_{i,0} (\sum\limits_{k=1}^\ell \lt_{i,k} \widetilde{\bu}_{i,k} \widetilde{\bu}_{i,k}^{\T})
\]
does not depend on a particular choice of   $\bu_{i,0} (\varepsilon),\ldots,\bu_{i,\ell} (\varepsilon)$.
\end{remark}
\begin{remark}[Limiting spectral properties from ASE]
From the ASE, it is easy to find to the asymptotic eigenvalues and eigenvectors of $\bKe$.
The orders of the eigenvalues, as well as their leading terms $\vect{\lt}_i$  can be extracted uniquely the ASE, by taking eigenvalues of $\Kb_i$.
There is, however, additional ambiguity which may arise for eigenvectors.
If it happens that the two leading terms in a block coincide (e.g., $\lt_{i,1} = \lt_{i,2}$),  we can retrieve only the corresponding invariant subspace from the eigenvalue decomposition of $\Kb_i$.
To lift such an ambiguity, we may need to continue  perturbation series to higher orders. 
\end{remark}
Let us now list a few properties of the ASE, most of which are very easy to prove. 
\begin{lemma}
  \label{lem:properties-ase}
  The ASE has the following properties.
  \begin{enumerate}
  \item Every term in $\Kbe$ is symmetric, i.e. $\Kb_i=\Kb_i^\T$  for all $i$. 
  \item The terms in $\Kbe$ are orthogonal, $\Kb_i^\T\Kb_j=0$ if $i \neq j$.
  \item If $\bKe$ has full rank for $\varepsilon>0$, then so does $\Kbe$.
  \item Let $\bQ$ an orthogonal matrix ($\bQ^{-1}=\bQ^\T$). Then $\underline{\bQ^\T \bKe
      \bQ} = \bQ^\T \Kbe \bQ$. 
  \end{enumerate}
\end{lemma}
\begin{proof}
  (1) and (2) follow directly from the definition.

  For (3), if $\Kbe$ is not
  full rank, then there is some $\bx$ such that $\Kbe\bx=0$. $\Ulim$ is full
  rank by construction, so $\Kbe\bx=0$ implies that at least one of the
  eigenvalues is $0$ for all $\varepsilon$, which contradicts the assumption that
  $\bKe$ is invertible for $\varepsilon>0$.

  (4) follows from applying the change of basis to $\bKe$, which leaves the
  eigenvalues intact but changes $\Ulim$ to $\bQ\Ulim\bQ^\T$.
\end{proof}

\begin{example}\label{ex:5x5-ase}
  We return to the matrix treated in \Cref{ex:5x5,ex:5x5-notation}. In this case we had three
  groups of eigenvalues, of order $1,\varepsilon^2,\varepsilon^4$,
  we get:
  \[ \Kbe =  \Kb_0 + \varepsilon^2 \Kb_1 + \varepsilon^4 \Kb_2= \Ulim_0 \Ulim_0^\T + \varepsilon^2 \Ulim_1
    \begin{pmatrix}
      2.08 & 0 \\
      0 & 0.57 
    \end{pmatrix} \Ulim_1^\T
    + \varepsilon^4 \Ulim_2
    \begin{pmatrix}
      2.18 & 0 \\
      0 & 0.77
    \end{pmatrix} \Ulim_2^\T.
  \]
  where all numerical values are truncated to two digits.
  
Then the ASE is equal to \[ \Kbe =
  \begin{pmatrix}
    1 & & & & \\
      & & & & \\
      & & & & \\
      & & & & \\
      & & & & \\
  \end{pmatrix}
  +
  \varepsilon^2
  \begin{pmatrix}
     & & & & \\
      & 0 & \frac{1}{2} & & \\
      & \frac{1}{2}& 1 & & \\
      & & & & \\
      & & & & \\
  \end{pmatrix}
  +
  \varepsilon^4
  \begin{pmatrix}
    & & & & \\
    &  &  & & \\
    & &  & & \\
    & & & \frac{1}{8}& \frac{1}{2} \\
    & & & \frac{1}{2} & 1
  \end{pmatrix}
\]  
  
\end{example}

Our method consists in obtaining formulas for the ASE, from which asymptotic
eigenvalues and eigenvectors can then be read out. 
In short, for the $k$-th block the limiting spectral information can be retrieved from  $\Kb_k$.   The following example explain this process as well as  the possible ambiguities that  arise when there are multiple eigenvalues. 


\subsection{Asymptotic equivalence up to an order}
\label{sec:ASE-equivalence}

The ASE lets us define a natural notion of equivalence between matrix
perturbations, which says that two matrices have identical limiting spectral
behaviour up to a certain order. Let   $\truncASE{\bKe}{q}$ be the truncated ASE, so that:
\begin{equation}
  \label{eq:trunc-ase}
  \truncASE{\bKe}{s}    = \sum_{i: \alpha_i \le s} \varepsilon^{\alpha_i} \Kb_i
= \Ulim \left( \trunc{s} \leadt(\bLa) \right)  \Ulim^\T,
\end{equation}
i.e., all the terms of valuation greater than $q$  are removed from the truncated ASE. 

\begin{definition}[Asymptotic equivalence at order $q$]
  Two matrices $\bKe$ and $\bKpe$ are asymptotically equivalent at order $s \in \ZZp$,
  noted $\bKe \equivq \bKpe$ if and only if
  $\truncASE{\bKe}{q} = \truncASE{\bKpe}{q}$.
  Two matrices are asymptotically equivalent at all orders, noted $\bKe
  \equivinf \bKpe$, if they are equivalent for all $q$, i.e. if
  $\ASE{\bKe} = \ASE{\bKpe}$.
\end{definition}
\begin{example}
  Let $\bA$ a full rank matrix. Then for all $\bB,\bC$, the perturbed matrices $\bKe = \bA +
  \varepsilon\bB$  and $\bKpe = \bA +
  \varepsilon\bC $ are equivalent to all orders, since $\ASE{\bKe} =
  \ASE{\bKpe}= \bA$. 
\end{example}

This lets us define equivalence classes, in the space of analytic perturbations,
of matrices that are ASE-equivalent up to order $q$ (we quotient by $\equivq$).
By definition, in each of these equivalence classes, all matrices have the same
$q$-truncated ASE, and the latter is a distinguished element, since its
$q$-truncated ASE is itself. 
\begin{proposition}
  \label{prop:trunc-pse-is-proj}
  Fix $\bKe$, $q$, and let $\calK_q = \left\{ \bKpe \vert \bKpe \equivq \bKe
  \right\}$. There is a unique matrix in $\bC(\varepsilon) \in \calK_q$ such
  that $\truncASE{\bC(\varepsilon)}{q} = \bC(\varepsilon)$, and this matrix
  equals $\bC(\varepsilon) = \truncASE{\bKpe}{q}$ for any $\bKpe \in \calK_q$.
\end{proposition}
\begin{proof}
  It is easy to verify that  $\truncASE{(\truncASE{\bKpe}{q})}{q} =
  \truncASE{\bKpe}{q} $ from the definition, eq.  \eqref{eq:trunc-ase}: the
  eigenvectors are unchanged and no eigenvalue has valuation $>q$.
  To check that no other matrix in the equivalence class has this property,
  suppose that there exists $\bC \in \calK_q$, such that $\bC \neq
  \truncASE{\bKe}{q}$ but $\truncASE{\bC}{q} = \bC$. Since $\bC \in \calK_q$, then  it must have the same
  limiting eigenvectors and its limiting eigenvalues must have the same leading
  terms. This implies that $\bC$ can only differ from $\truncASE{\bKe}{q}$ in
  some other way, but in this case $\truncASE{\bC}{q} \neq \bC$ and we have a
  contradiction. 
\end{proof}

For what follows, it is helpful to give an equivalent characterisation of these
distinguished elements:
\begin{proposition}
  \label{prop:how-to-recognise-an-ASE}
  A matrix $\bC(\varepsilon)$ is a $p$-truncated ASE if and only if there exists
  an orthonormal $\bQ \in \R^{n \times n}$ such that:
  \[ \bQ^\T \bC(\varepsilon) \bQ= \bD(\varepsilon) \]
  where $\bD$ is a diagonal matrix where each entry is a monomial
  in $\varepsilon$, i.e.  $\leadt D_{i,i}(\varepsilon) = D_{i,i}(\varepsilon)$, 
  and the maximum degree of these polynomials is $p$. 
\end{proposition}
\begin{proof}
  This follows directly from the definition of the truncated ASE, see eq.
  \eqref{eq:trunc-ase}.  
\end{proof}

\subsection{Regularised inverses and their asymptotics}
\label{sec:reg-inverse}

  One of the key tools in our proofs are regularized inverses which serve as a probing device to obtain ASEs.
  The ``regularised inverse''  of a matrix $\bKe$ is the matrix
  \[ 
  \bM(z,\bK) = \bKe(\bKe+z\bI)^{-1} = \bI - (\bKe + z \bI)^{-1},
  \]
  defined for those $z$ such that $\bKe+z\bI$ is invertible. Regularised
  inverses often turn up in the theory of kernel methods in statistics
  \cite{barthelme2021determinantal}, and as the identity above makes clear, are
  strongly related to the resolvent. Unlike traditional approaches in analytic
  perturbation theory, we use regularised inverses in a way that sidesteps the
  need for complex analysis. Our strategy consists in setting $z = \tau
  \varepsilon^s$, and taking the limit $\flatlim$ for various values of $s$. As
  we show in \Cref{lem:RI}, this filters out in $\bM(\tau \varepsilon^s
  ,\bK)$ the parts of the eigenspace with valuation greater than $s$.
  
  What is important for our purposes is that regularised inverses are in fact
  sufficient to characterise the limiting spectral behaviour of $\bKe$. We shall
  use the following definition:
  \begin{definition}[Limit of regularised inverse]
    \label{def:limit-reg-inv}
    For $s \in \ZZp$, we define the following limit:
      \begin{equation}
    \label{eq:lim-reg-inv}
    \bM_{s,\tau}(\bK) = \lim_{\flatlim} \bKe(\bKe + \tau \varepsilon^s \bI)^{-1}
  \end{equation}
  for all $\tau$ such that the limit exists (i.e. all $\tau \notin
  \{-\lt_{ij}\} $).
  \end{definition}
  
  With this definition, we can define a necessary and sufficient condition for
  asymptotic equivalence based on limits of regularised inverses: 
  \begin{proposition}[Regularised inverses and asymptotic equivalence]
  \label{prop:RI-asymp-equiv}
  $\bKe \equivq \bKpe $ if and only if
  \begin{equation}
    \label{eq:lim-reg-inv-equiv}
    \bM_{s,\tau}(\bK) = \bM_{s,\tau}(\bK')
  \end{equation}
  for all $s \leq  q$ and any valid $\tau$. 
\end{proposition}
We defer the proof for now. The way we use the lemma is as follows: for certain
classes of matrices, we can directly compute $\bM_{s,\tau}(\bK)$ for different
values of $s$. Because these limits take simple forms, we can easily exhibit a
matrix $\bC(\varepsilon)$ that has the form of a ASE, and such that
$\bM_{s,\tau}(\bK) = \bM_{s,\tau}(\bC) $, and we conclude from 
\Cref{prop:RI-asymp-equiv} that $\bC$ is in fact the ASE (or the ASE truncated to
a certain order).

\Cref{prop:RI-asymp-equiv} follows from the following lemma, which shows
that $\bM_{s,\tau}(\bK)$ depends only on the ASE of $\bK$. By increasing $s$, we
can increase the span of $\bM_{s,\tau}(\bK)$ to include successive blocks of
limiting eigenvectors. 
\begin{lemma}[Asymptotics of regularised inverses]
  \label{lem:RI}
  Let $\bKe$ a symmetric matrix with asymptotic spectral equivalent $\Kbe = \sum_{i=0}^p
  \varepsilon^{\alpha_i} \Kb_i$. Let $s>0$, $j = \underset{s.t.\ \alpha_i \leq s}{\argmax
    \ i}$.
  
  Then for any $\tau \notin  \{-\lt_{ij}\},s \in \ZZp $,    $\bM_{s,\tau}(\bK)$
  is completely determined by the ASE to order $s$: 
  \begin{equation}
    \label{eq:reginv}
    \bM_{s,\tau}(\bK) = \bM_{s,\tau}(\truncASE{\bK}{s}) =
    \begin{cases}
      \sum_{i=0}^j \Ulim_{i}\Ulim_{i}^\T, & \mathrm{\ if\ } \alpha_j < s, \\
      \sum_{i=0}^{j-1} \Ulim_{i}\Ulim_{i}^\T + \Kb_j(\Kb_j + \tau \bI)^{-1} , & \mathrm{\ otherwise}.
    \end{cases}
  \end{equation}
\end{lemma}

\begin{proof}
  By \Cref{lem:rellichs-theorem}  and from \eqref{eq:normalization_U} we have that
  \begin{align*}
    \label{eq:reginv-kato}
    \bM(\tau \varepsilon^s, \bK) &\eqdef \bKe(\bKe + \tau \varepsilon^s )^{-1} = \bU(\varepsilon)\bLa(\varepsilon)(\bLa(\varepsilon) + \tau \varepsilon^s \bI)^{-1}\bU(\varepsilon)^\T \\
    &= (\Ulim + o(1)) \bLa(\varepsilon)(\bLa(\varepsilon) + \tau \varepsilon^s \bI)^{-1} (\Ulim^\T + o(1)),
  \end{align*}
  so we can focus on the limit of the middle term. The matrix
  $\bLa(\varepsilon)(\bLa(\varepsilon) + \tau \varepsilon^s \bI)^{-1}$ is
  diagonal, and we are going to find the limits of its diagonal elements.
  Consider an eigenvalue $\lambda_{i,k}(\varepsilon)$ in the $i$-th group
  \eqref{eq:eig_expansion_individual} . Then, if $\tau \neq -\lt_{i,k}$, the
  corresponding diagonal element of $\bLa(\varepsilon)(\bLa(\varepsilon) + \tau
  \varepsilon^s \bI)^{-1}$ is equal to
  \begin{equation}
    \label{eq:reginv-ev}
    \frac{\lambda_{i,k}(\varepsilon)}{\lambda_{i,k}(\varepsilon)+\tau\varepsilon^s} = 
    \begin{cases}
      1 + o(1), & \mathrm{\ if\ } s > \alpha, \\
      \frac{\lt_{i,k}}{\tau+\lt_{i,k}}  + o(1), & \mathrm{\ if\ } s = \alpha, \\
      0+ o(1), & \mathrm{\ otherwise},
    \end{cases}
  \end{equation}
  where these three cases follow from the power series  expansion of $\frac{1}{1+x}$ at 0.
  
   Injecting eq. \eqref{eq:reginv-ev} into eq. \eqref{eq:reginv-kato}, we find the following:
   \begin{itemize}
   \item   if $s$ does not match any of the valuations ($\alpha_j \neq s$),
  \[ \bM_{s,\tau} = \underset{\flatlim}{\lim} \sum_{i=0}^j \Ulim_i (\bI_{c_i} +
    \diag(o(1)))\Ulim_i^\T = \sum_{i=0}^j \Ulim_i\Ulim_i^\T , \]
   where $c_i$ is the size of the block of eigenvalues with
  valuation $\alpha_i$; 
  \item if  $\alpha_j=s$, then by denoting 
  \[
  \widehat{\Lambda}_j = \diag\left(\frac{\lt_{j,1}}{\lt_{j,1}+\tau}, \ldots,\frac{\lt_{j,c_j}}{\lt_{j,c_j}+\tau} \right),
  \]
  we have that
 \[
 \bM_{s,\tau} = \sum_{i=0}^{j-1} \Ulim_i   \Ulim_i^\T +
    \Ulim_j(  \widehat{\Lambda}_j)\Ulim_j^\T.
    \]
 Elementary calculations show that $\Ulim_j  \widehat{\Lambda}_j  \Ulim_j^\T =
 \Kb_j(\Kb_j + \tau \bI)^{-1} $.
 Finally, showing that $\bM_{s,\tau}(\bK) = \bM_{s,\tau}(\truncASE{\bK}{s})$
 follows directly from eq. \eqref{eq:reginv-ev} and the definition of the
 truncated ASE (eq. \eqref{eq:trunc-ase}). 
  \end{itemize}
\end{proof}

We can now prove \Cref{prop:RI-asymp-equiv}.
\begin{proof}[Proof of \Cref{prop:RI-asymp-equiv}]
  One direction is straightforward: if $\bKe \equivq \bKpe$, then their
  truncated ASEs are equal up to order $q$, and so $\bM_{s,\tau}(\bK) =
  \bM_{s,\tau}(\truncASE{\bK}{s}) = \bM_{s,\tau}(\truncASE{\bK'}{s})=
  \bM_{s,\tau}(\bK') $ for all $s\leq q$ by \Cref{lem:RI}. In the other
  direction, we have to check that $\bM_{s,\tau}(\bK) =\bM_{s,\tau}(\bK') $ for
  all $s\leq q$ implies that the ASEs of $\bK$ and $\bK'$ are equal up to order
  $q$. This holds because all terms of the ASE up to order $p$ can be
  reconstructed from the sequence $\bM_{s,\tau}(\bK)$ for $s \in \{0, \ldots, q\}$. To
  see why this is the case, consider that for all $s$ such that $s$ matches a valuation
  $\alpha_j$ for the $j$-th group of eigenvalues, eq. \eqref{eq:reginv}, states
  that $\bM_{s,\tau}(\bK)$ is the sum of a projection matrix on all groups of
  eigenvectors of index $< j$, plus a matrix of the form $\Kb_j(\Kb_j + \tau
  \bI)^{-1}$. The latter is in one-to-one correspondance with $\Kb_j $. Thus,
  all terms of order up to $q$ in the ASE can be reconstructed from the sequence
  of regularised inverses.
\end{proof}

We are now ready to present our main results. 

\section{Matrices in diagonal-scaling form}
\label{sec:diagonal-scalings}

The reader may have noticed that in many examples we have given so far the
entries of the matrix are of different orders: some entries of order $1$, some
entries of order $\varepsilon$, some entries of order $\varepsilon^2$, etc. When such
structure exists, it can be exploited to obtain asymptotic expansions more
easily (an idea that appears in the tropical algebra literature, see \cite{akian2004min}).

\begin{example}
    We return to  \Cref{ex:5x5}, which we reproduce here for convenience
  \[
  \bKe =
  \begin{pmatrix}
    1 & \frac{\varepsilon}{2} & \varepsilon^4 & 0 & 0 \\
    \frac{\varepsilon}{2} & \frac{1}{4} \varepsilon^2 & \frac{\varepsilon^2}{2} & 0 & 0\\
    \varepsilon^4 & \varepsilon^2 & \varepsilon^2 & \frac{\varepsilon^3}{2} & 0 \\
    0  & 0 &  \frac{\varepsilon^3}{2} & \frac{1}{8} \varepsilon^4 & \frac{\varepsilon^4}{2} \\
    0 & 0 & 0 & \frac{\varepsilon^4}{2} & \varepsilon^4 
  \end{pmatrix}
\]
Not only are the entries of different magnitudes in $\varepsilon$, but here they
are ordered such that the valuation is non-decreasing across rows and columns
($\val K_{i+c,j} \geq \val K_{i,j}$ for $c\geq 0$).
We can take advantage of this property to re-express $\bKe$
as:
\begin{equation*}
  \bKe = \bDe(\bDe^{-1} \bK \bDe^{-1})\bDe =  \bDe \begin{pmatrix}
    1 & \frac{1}{2} & \varepsilon^3 & 0 & 0 \\
    \frac{1}{2} & \frac{1}{4}  & \frac{1}{2} & 0 & 0\\
    \varepsilon^3 & \frac{1}{2} & 1 & \frac{1}{2} & 0 \\
    0  & 0 &  \frac{1}{2} & \frac{1}{8} & \frac{1}{2} \\
    0 & 0 & 0 & \frac{1}{2} & 1
  \end{pmatrix} \bDe
\end{equation*}
where
$\bDe=\diag(\varepsilon^0,\varepsilon^1,\varepsilon^1,\varepsilon^2,\varepsilon^2)$
is a \emph{diagonal scaling matrix}. In this case, before diagonal scaling, the
matrix $\bKe$ has a single entry of valuation 0. After diagonal scaling, the
matrix $\bDe^{-1} \bK \bDe^{-1}$ can be rewritten as $\bDe^{-1} \bK
\bDe^{-1}=\bH+\O(\varepsilon^3)$, where $\bH$ is a full-rank matrix with entries
of valuation 0. \Cref{thm:ASE-H} below describes the asymptotic spectral
behaviour of $\bKe$ in terms of the scaling matrix $\bDe$ and the matrix $\bH$. 
\end{example}

For matrices which admit a non-trivial diagonal scaling, we can derive precise results on the asymptotic eigenvalues and eigenvectors (\Cref{thm:ASE-H}).
In the next subsection, we set the notation and assumptions and formulate the main result of this section. 
To make the result more user-friendly,  \cref{sec:how-to-use} explains how to find and use diagonal scalings in computations 
and highlight links to tropical algebra.
The proof of \Cref{thm:ASE-H} is presented in \cref{sec:reg-inverse-diag-scaled}, and relies on regularised inverses. 

\subsection{ASE of diagonally-scaled matrices}
\label{sec:notation-scalings}
In this section, we address the following case: 
\begin{equation}
  \label{eq:H-scaling}
  \bKe = \bDe(\bH+o(1))\bDe
\end{equation}
where $\bDe$ is a scaling matrix. Any matrix perturbation can be put into the
form of eq. \eqref{eq:H-scaling}, if only under the trivial scaling
$\bDe = \bI$. In such a case our theorem will have nothing of much interest to
say - it will only describe the eigenvalues and eigenvectors of order
$\varepsilon^0$. Our results begin to be interesting if the scaling matrix is
non-trivial, which means that $\bKe$ needs to have entries with different orders
of magnitude in $\varepsilon$. 

We need to set up some notation to describe the scaling matrix $\bDe$ and the
block structure it induces in $\bH$. We write:
\begin{equation}\label{eq:Delta-scaling} 
\bDe= \matr{\Delta_b} (\varepsilon) = \diag( \underbrace{\varepsilon^{\nu_0}, \dots, \varepsilon^{\nu_0}}_{b_0}, \underbrace{\varepsilon^{\nu_1}, \ldots, \varepsilon^{\nu_1}}_{b_1}, \ldots, \underbrace{\varepsilon^{\nu_p}, \dots, \varepsilon^{\nu_p}}_{b_p})
\end{equation}
where  $\nu_0 < \nu_1 < \cdots < \nu_r$ and  each
valuation $\nu_i$  (not necessarily integer) is repeated $b_i$ times.

We partition $\bH$ according to the valuations in $\bDe$, as
\begin{equation}
\bH =
  \begin{pmatrix}
    \bH_{0,0} & \bH_{0,1} & \dots & \bH_{0,p} \\
    \bH_{1,0} & \bH_{1,1} & \dots & \bH_{1,p}\\
    \vdots & \vdots & \dots & \dots \\
    \bH_{p,0} & \bH_{p,1} & \dots & \bH_{p,p}
  \end{pmatrix},
\end{equation}
with $\bH_{i,j} \in \RR^{b_i \times b_j}$ ($b_i$ are positive integers).
For convenience, we define:
\begin{itemize}
\item  $\bH_{\leq i,\leq j }$ to be the submatrix of $\bH$ with row blocks up to $i$ and column blocks up to $j$;
\item  $\bH_{i,\leq j }$ to be the submatrix of the $i$-th block row 
\[
\bH_{i,\leq j } =   \begin{pmatrix}  \bH_{0,0} & \bH_{0,1} & \dots & \bH_{0,j}\end{pmatrix},
\]
and similar notation  $\bH_{\leq i,j}$ for the submatrix of the $j$-th block column;
\item shortcuts $\bH_{<i,<j} =\bH_{\le i-1,\le j-1}$, $\bH_{i,<j} =\bH_{i,\le j-1}$, $\bH_{<i,j} =\bH_{\le i-1,j}$.
\end{itemize}
Whenever $\bH_{<i,<i}$ is invertible, we define the $i$-th Schur complement as
\begin{equation}\label{eq:ScH}
\bS_i = \bH_{i,i}-\bH_{i,<i} (\bH_{<i,<i})^{-1}\bH_{<i,i},
\end{equation}
and formally define $\bS_0=\bH_{0,0}$. If there exists a $j$ such that $\bS_{j}$ is not invertible, then
the sequence of Schur complements stops at this $\bS_{j}$ (non-invertibility of
$\bS_{j}$ implies non-invertibility of $\bH_{< j+1, < j+1}$).

Armed with the above notation, we can formulate the following theorem.
\begin{theorem}
  \label{thm:ASE-H}
  Let $\bKe$ be as in eq. \eqref{eq:H-scaling}. Assume that the matrix $\bH_{\le
    j, \le j}$, for $j \le p$  is invertible (i.e., all Schur complements up to
  $ \bS_j$ exist). Then:
\begin{enumerate}
\item  the truncated ASE of $\bKe$  has the following block-diagonal form, to
  order $2 \nu_j$
  
\begin{equation}\label{eq:ASE-H}
 \truncASE{\bK}{2\nu_j}(\varepsilon) =
  \begin{pmatrix}
   \varepsilon^{2\nu_0} \bS_0 &  & &  \\
    & \ddots &&\\
     &  & \varepsilon^{2\nu_j} \bS_j &  \\
   & &  & \matr{0}
  \end{pmatrix},
\end{equation}
\item if $\bH$ is invertible (i.e., $j=p$), the ASE is completely determined by
  $\bH$:
\begin{equation}\label{eq:ASE-H-full}
 \Kb(\varepsilon) =
  \begin{pmatrix}
   \varepsilon^{2\nu_0} \bS_0 &  & &  \\
    &   \varepsilon^{2\nu_1} \bS_1 &&\\
     &  & \ddots &  \\
   & &  & \varepsilon^{2\nu_p} \bS_p
  \end{pmatrix},
\end{equation}
\end{enumerate}
\end{theorem}
\begin{proof}
  The proof of  \Cref{thm:ASE-H} is deferred to \cref{sec:reg-inverse-diag-scaled}, and uses regularised inverses.
\end{proof}

\begin{remark}[Related results]
  We know of related, but not equivalent, results in the literature. 
  Tropicalisation of the characteristic polynomial of $\bKe$ can be used to
  lower-bound the valuation of the eigenvalues \cite{akian2014tropical,
    akian2004min, akian2016non}, and obtain the leading coefficients of
  eigenvalues in certain cases.  Schur complements appear in the Lidskii--Vishik--Lyusternik
  approach to perturbation theory of non-symmetric matrices
  \cite{lidskii1966perturbation,moro1997lidskii}, and in \cite{carlsson2018perturbation}.
\end{remark}
\begin{remark}[Localization]
  \Cref{thm:ASE-H} implies that the eigenvectors of diagonally-scaled
  matrices are asymptotically \emph{localised}, meaning that they are non-zero
  over only part of the domain. Determining classes of matrices in which
  eigenvectors are localised is an important topic in applications
  \cite{benzi2016localization}. 
\end{remark}
How to work with and interpret the results of  \Cref{thm:ASE-H} will
hopefully become clearer with the tools we introduce in \cref{sec:how-to-use}. First, let us explain how the asymptotic eigenvectors and
eigenvalues can be recovered from the ASE, as given in eq. \eqref{eq:ASE-H} in
the general case.
  
  Take $j$ the maximum index such that $\bS_j$ exists, and denote for each $i \le j$ the block matrix 
 \[
\bZ_i =\begin{pmatrix}\zeroes & \cdots & \bI_{b_i} & \cdots & \zeroes \end{pmatrix}^{\T}.
\]
Note that in this case eqs.  \eqref{eq:ASE-H} and \eqref{eq:ASE-H-full}, can be written as 
\begin{align}
 \Kb(\varepsilon)& = \sum_{i=0}^j \varepsilon^{2\nu_i} \bZ_i \bS_i \bZ_i^\T  + o(\varepsilon^{2\nu_j }), \quad  \text{and }\label{eq:ASE-H-Z} \\
 \Kb(\varepsilon) &= \sum_{i=0}^p \varepsilon^{2\nu_i} \bZ_i \bS_i \bZ_i^\T, \label{eq:ASE-H-full-Z}
 \end{align}
respectively.
 
For $i < j$,  \Cref{thm:ASE-H} states that there exists a block of eigenvalues with valuation $2 \nu_i$. Since $\bZ_i$ is orthonormal, the eigenvalues and eigenvectors of the term $\bZ_i
  \bS_i \bZ_i^\T$ in the ASE can be obtained from the eigenvalues and
  eigenvectors of $\bS_i$. By assumption, $i < j$ so that $\bS_i$ is invertible, and has
  $b_i$ non-zero eigenvalues $\eta_{i,1}\dots \eta_{i,b_i}$. Then $\bKe$ has a
  block of asymptotic eigenvalues of the form $\lambda(\varepsilon) =
  \varepsilon^{2\nu_i}(\eta_{i,k} + \O(\varepsilon))$ for $k \in \{1, \ldots, b_i\}$.
  The corresponding eigenvectors can be obtained from the eigenvectors of
  $\bS_i$, if all eigenvalues of $\bS_i$ are simple. If there are repeated
  eigenvalues in $\bS_i$, then the asymptotic eigenvectors cannot be identified
  (a further expansion is needed to make the eigenvalues distinct).

  In interpreting the last block $j$ in the expansion of the ASE, we need to be
  careful. If $j < p$ then the expansion is truncated early, we are in case 1 of
  the theorem, and the ASE is only identified up to valuation $2\nu_j$. The
  Schur complement $\bS_j$ is non-invertible. Its non-zero eigenvalues (and
  eigenvectors) give the leading coefficients of asymptotic eigenvalues of
  $\bKe$ with valuation $2\nu_j$, and its zero eigenvalues (and the associated
  null space) correspond to eigenvalues of $\bKe$ with valuation strictly higher
  than $2\nu_j$.
  If, on the other hand, $\bH$ is invertible, then the last Schur
  complement $\bS_j$ has full rank, and all asymptotic eigenvalues of $\bKe$ can
  be identified using the theorem.

\subsection{\Cref{thm:ASE-H}: a user's guide}
\label{sec:how-to-use}

To make \Cref{thm:ASE-H} more useful in calculations (either by hand or
on a computer), let us explain how to easily compute $\bH$ from $\bKe$ given a
candidate scaling. The right tools to use come from tropical algebra (see e.g.
\cite{joswig2021essentials} for an introduction). Fortunately, they are 
easy to understand and can be described with minimal background.

In computations by hand, it is useful to write down a \emph{valuation matrix}
for the entries of $\bKe$:
\begin{definition}[Valuation matrix]
  \label{def:valuation-matrix}
  The valuation matrix $\vM$ of $\bKe$ is a matrix with entries in $\ZZp \cup
  \{\infty \}$  that contains
  the element-wise valuations of $\bKe$, i.e.
  $\vM = [\val K_{ij}(\varepsilon)]_{i,j} $
\end{definition}
\begin{example}
  \label{ex:scaling-5x5}
  We return once again to the 5x5 matrix of \Cref{ex:5x5}. Its valuation matrix is:
\[ \vM =
  \begin{pmatrix}
    0 & 1 & 4 & \infty & \infty \\
    1 & 2 & 2 & \infty & \infty \\
    4 & 2 & 2 & 3 &\infty \\
    \infty & \infty & 3 & 4 & 4 \\
    \infty & \infty & \infty & 4 & 4 \\
  \end{pmatrix}
\]
  Recall that $\val(0) = \infty$.
\end{example}

Diagonal scalings need to be designed carefully so that:
\begin{equation}
  \label{eq:diag-scaling-repeated}
  \bKe = \bDe(\bH + o(1))\bDe
\end{equation}
as in \Cref{thm:ASE-H}. Equation \eqref{eq:diag-scaling-repeated}
implies:
\begin{equation}
  \label{eq:diag-scaling-bound}
  \val K_{ij}(\varepsilon) \geq \val \Delta(\varepsilon)_{ii} + \val \Delta(\varepsilon)_{jj}
\end{equation}
This leads to the following definition:
\begin{definition}[Valid scaling]\label{def:valid-scaling}
  We say $\bDe$ is a valid scaling for $\bKe$ if eq.
  (\ref{eq:diag-scaling-bound}) is verified for all $i,j$.
  We say that $\bDe$  is \emph{tight} at entry $i,j$ if:
  \[ \val K_{ij}(\varepsilon) = \val \Delta(\varepsilon)_{ii} + \val \Delta(\varepsilon)_{jj}\]
\end{definition}

Given a candidate scaling $\bDe$, one can form the
matrix $\vMt$ with entries
\[ \tilde{\Omega}_{i,j} = \val \Delta(\varepsilon)_{ii} + \val \Delta(\varepsilon)_{jj}   \]
Then $\bDe$ is a valid scaling iff
\begin{equation}
  \label{eq:validity-scaling}
  \vM - \vMt \geq 0
\end{equation}
element-wise. There is a systematic way of finding valid scalings that are
maximally tight, via the Hungarian algorithm, see e.g.
\cite{hook2019max,joswig2021essentials} for an introduction.

Once a valid scaling has been found, the matrix $\bH$ is easy to compute:
\begin{proposition}
  \label{prop:H-from-tightness}
  Given a valid scaling $\bDe$, the decomposition
  \[   \bKe = \bDe(\bH + o(1))\bDe \]
  is verified for
  \begin{equation}
    \label{eq:H-tight}
    H_{ij} =
    \begin{cases}
      \leadc K_{ij}(\varepsilon) \textrm{\ if } \bDe \textrm{\ is tight at\ } (i,j)  \\
      0 \textrm{\ otherwise}
    \end{cases}
  \end{equation}
  Recall that $\leadc K_{ij}$ is notation for the leading coefficient of the
  entry $K_{ij}$. 
\end{proposition}
\begin{proof}
  Follows directly by verifying that eq. \eqref{eq:diag-scaling-repeated} holds
  entry-wise for both tight and non-tight entries.
\end{proof}

\begin{example}[ \Cref{ex:5x5,ex:scaling-5x5}, continued]
\label{ex:5x5-continued}
We can compute the asymptotics of the matrix given in \Cref{ex:scaling-5x5} using the
tools of this section. Some of the computations are best done with the help of a
Computer Algebra System. 
Take $\bDe =
\diag(\varepsilon^0,\varepsilon^1,\varepsilon^1,\varepsilon^2,\varepsilon^2
)$. Then, in the notation of \Cref{def:valid-scaling}:
\[
  \vMt =
  \begin{pmatrix}
    \he{0} & \he{1} & 1 & 2 & 2 \\
    \he{1} & \he{2} & \he{2} & 3 & 3 \\
    1 & \he{2} & \he{2} & \he{3} & 3 \\
    2 & 3 & \he{3} & \he{4} & \he{4} \\
    2 & 3 & 3 & \he{4} & \he{4}
  \end{pmatrix}
\]
One can check that the scaling is valid by computing $\vM - \vMt$, which should
have non-negative entries. We highlight in blue and italic the entries for which
the scaling is \emph{tight} ($\tilde{\Omega}_{i,j} = \Omega_{i,j}$). Notice that
increasing the valuation of any of the entries in $\bDe$ makes the scaling
invalid. Decreasing the valuation keeps the scaling valid, but fewer 
entries are tight.

The next step is to find $\bH$ in eq. \eqref{eq:diag-scaling-repeated}, which we
can do by applying \Cref{prop:H-from-tightness}:
\begin{equation*}
  \label{eq:}
  \bH =
  \left(
    \setstretch{1.25}
    \begin{array}{c|cc|cc}
      1 & \frac{1}{2} & & &  \\ 
      \hline
      \frac{1}{2} & \frac{1}{4} &  \frac{1}{2} & & \\
        & \frac{1}{2} & 1 &  \frac{1}{2} &  \\
      \hline 
        & &  \frac{1}{2} & \frac{1}{8} &  \frac{1}{2}  \\
        & &  &    \frac{1}{2} & 1 \\
    \end{array}
  \right)
\end{equation*}
The block structure in $\bH$ corresponding to the valuations is highlighted. The
Schur complements for this block structure are
\[ \bS_0 = 1,\quad \bS_1 =
  \begin{pmatrix}
    \frac{1}{4} &  \frac{1}{2} \\
    \frac{1}{2} & 1 
  \end{pmatrix}
  -
  \begin{pmatrix}
    \frac{1}{2} \\
    0
  \end{pmatrix} 1^{-1}
  \begin{pmatrix}
    \frac{1}{2} &   0
  \end{pmatrix}
=
  \begin{pmatrix}
    0 & \frac{1}{2} \\
    \frac{1}{2} & 1
  \end{pmatrix}, \quad 
\]
and
\begin{align*}
  \bS_2 = 
  \begin{pmatrix}
    \frac{1}{8} & \frac{1}{2} \\
    \frac{1}{2} & 1
  \end{pmatrix} -
  \begin{pmatrix}
    \frac{1}{2} \\ 0
  \end{pmatrix}
  \begin{pmatrix}
    1 & \frac{1}{2} & 0 \\ 
    \frac{1}{2} & \frac{1}{4} &  \frac{1}{2}  \\
      0  & \frac{1}{2} & 1 
  \end{pmatrix}^{-1}
  \begin{pmatrix}
     \frac{1}{2} & 0 
  \end{pmatrix}=   \begin{pmatrix}
    \frac{1}{8} & \frac{1}{2} \\
    \frac{1}{2} & 1
  \end{pmatrix}
\end{align*}
Then we can easily see that the ASE given by \Cref{thm:ASE-H} agrees with the one given in \Cref{ex:5x5-ase}, and thus the limiting eigenvalues and eigenvectors can be computed by diagonalising
$\bS_0,\bS_1$ and $\bS_2$ using the classical formulas for $2 \times 2$
matrices (we do not detail these calculations).
\end{example}
\begin{example}
  \label{ex:3x3}
  Our final example for this section concerns a matrix for which \Cref{thm:ASE-H}
  fails to characterise all eigenvalues, even with the optimal scaling (we end up in case (1) of the theorem).
  Take the  matrix $\bK$ with valuation $\vM$:
  \[ \bK =
    \begin{pmatrix}
      1 & \varepsilon & \varepsilon \\
      \varepsilon & \varepsilon^3 & -\varepsilon^3 \\
      \varepsilon & -\varepsilon^3 & \varepsilon^3 \\
    \end{pmatrix}, \quad \vM =
    \begin{pmatrix}
      0 & 1 & 1 \\
      1 & 3 & 3 \\
      1 & 3 & 3
    \end{pmatrix}.
  \]
  With the scaling $\bD = \diag(\varepsilon^0,\varepsilon^1,\varepsilon^1)$, we
have    \[ \vMt =
    \begin{pmatrix}
      \he{0} & \he{1} & \he{1} \\
      \he{1} & 2 & 2 \\
      \he{1} & 2 & 2
    \end{pmatrix}.
  \]
The tight entries are highlighted in blue, and it can easily be checked
  that the scaling cannot be improved.
  \Cref{prop:H-from-tightness} gives:
  \[
    \bH =
    \left( 
    \begin{array}{c|cc}
      1 & 1 & 1\\
      \hline
      1 & 0 & 0 \\
      1 & 0 & 0
    \end{array}
  \right)
  \]
  The Schur complements are $\bS_0=1$, $\bS_1 = - 
  \left(\begin{smallmatrix}
    1 & 1 \\
    1 & 1
  \end{smallmatrix}\right)
  $. Since $\bS_1$ is of rank one, there is a single 
  eigenvalue of valuation $2$. \Cref{thm:ASE-H} gives:
  \[
    \Kbe =
    \begin{pmatrix}
      1 & & \\
        &0 & \\
      & & 0\\
    \end{pmatrix}
    - \varepsilon^2
    \begin{pmatrix}
      0 & & \\
        & 1 & 1 \\
        & 1 & 1\\
    \end{pmatrix}
    + \O(\varepsilon^3)
  \]
  The last eigenvalue has valuation $> 2$ but is not identified. For this we
  need a stronger theorem, specifically 
  \Cref{thm:ASE-generalised-kernel-form}. We revisit this computation in 
  \Cref{ex:3x3-revisited}.
\end{example}
The rest of this section contains the proof of \Cref{thm:ASE-H}, and
readers can skip ahead to  \cref{sec:generalised-kernel-matrices} for a
generalisation of the theorem that is more widely applicable.

\subsection{Proof of the main theorem}
\label{sec:reg-inverse-diag-scaled}

We are now ready to give the proof of \Cref{thm:ASE-H}, which is based on the properties of ASE and 
regularised inverses, as introduced in \cref{sec:reg-inverse}. 
In particular, we show that the regularised inverses inherit a block structure from the diagonal scaling,
which leads to the required result.

\begin{proof}[Proof of {\Cref{thm:ASE-H}}]
With some abuse of notation, let $\rInvK{s}{\tau}$ denote the regularized inverse.
First, let us note that use an alternative form
\[
\rInvK{s}{\tau} = \bKe(\bKe + \tau\varepsilon^s\bI)^{-1}= \bI - \tau \varepsilon^s(\bKe +\tau\varepsilon^{s}\bI)^{-1},
\]
which we prefer in the proof because it is symmetric. 
Now, let us take a particular $s \ge 0$  and compute directly the leading term of  $\rInvK{s}{\tau}$.
For a given $s$, define
\[
\ell = \ell(s) = \underset{\nu_i \leq \frac{s}{2}}{\max i},
\]
to be the index of the last block in $\bDe$ to have valuation less than or equal to $\frac{s}{2}$.
We consider two cases:
\begin{itemize}
\item \emph{matching}: i.e.,  $\nu_\ell = \frac{s}{2}$, i.e. $\frac{s}{2}$ matches one of the valuations $\nu_i$;
\item \emph{non-matching}: in which case $\nu_\ell<\frac{s}{2}$.
\end{itemize}
To compute the leading term of the regularized inverse, we define a modified diagonal scaling where each valuation greater than $\nu_\ell$ is clipped to $\frac{s}{2}$:
\begin{equation}
\label{eq:modified-scaling}
\bDet=\diag( \beps^{\min{(\nu_0,\frac{s}{2})}}, \dots, \dots, \beps^{\min{(\nu_r,\frac{s}{2})}} )
\end{equation}
This is still a valid diagonal scaling \eqref{eq:H-scaling} (although not necessarily optimal) so that $\bKe = \bDet (\bHt + o(1)) \bDet$, 
where the coefficient matrix has the following expression:
\[
\bHt =
\begin{pmatrix}
\bH_{\leq \ell, \leq \ell} & \zeroes \\
\zeroes  &  \zeroes
\end{pmatrix}.
\]
Injecting the definition of $\bKe$  leads to:
\begin{equation}\label{eq:reginv-dscaling-symmetric}
\rInvK{s}{\tau} = \bI - \tau \varepsilon^s \left( \bDet(\bHt+o(1))\bDet + \tau\varepsilon^s \bI \right)^{-1}.
\end{equation}
Next, by putting  $\tau\varepsilon^s$  inside the brackets, we  can bring  \eqref{eq:reginv-dscaling-symmetric} into the form \begin{equation}\label{eq:reginv-mod-scaling}
\rInvK{s}{\tau} = \bI - \tau \left(\bDet(\bHh+o(1))\bDet\right)^{-1},
\end{equation}
where the exact shape of  the coefficient $\bHh$ depends on whether we are in the matching or non-matching case:
\[
\bHh =
\begin{cases}
\begin{pmatrix}
\bH_{\leq \ell, \leq \ell} & \zeroes \\
\zeroes  & \tau \bI
\end{pmatrix}, & \mathrm{\ if\ } \nu_\ell<\frac{s}{2}, \\
\begin{pmatrix}
\bH_{< \ell, < \ell} & \bH_{< \ell,\ell} & \zeroes \\
\bH_{ \ell, <\ell } &   \bH_{\ell,\ell}+\tau \bI & \zeroes \\
\zeroes  & \zeroes & \tau \bI
\end{pmatrix}, & \nu_\ell = \frac{s}{2}.
\end{cases}
\]
Next, we use the fact that $\bDet$ is square to pull it out of the inverse in eq. \eqref{eq:reginv-mod-scaling}, and obtain:
\begin{equation}
\label{eq:reginv-mod-scaling-2}
\rInvK{s}{\tau} = \bI - \tau\varepsilon^{\frac{s}{2}}\bDet^{-1}\left( \bHh + o(1) \right)^{-1}\bDet^{-1}\varepsilon^{\frac{s}{2}}
\end{equation}
and we note that block $i$ in $\bDet^{-1}\varepsilon^{\frac{s}{2}}$ is either
$o(1)$ if $\nu_i<  \frac{s}{2}$, or $1+o(1)$ otherwise. 
Note that the matrix  $\bHh$ is invertible under the assumptions in the theorem (i.e., invertibility of $\bH_{\le \ell,\le \ell}$ for the non-matching case and invertibility of both $\bH_{<\ell,<\ell}$ and $\bS_\ell + \tau \bI_{c_\ell}$ for the matching case),  so that:
\begin{equation}
  \label{eq:reginv-mod-scaling-3}
\rInvK{s}{\tau} = \bI - \tau 
    \begin{pmatrix}
      \zeroes & \zeroes \\
      \zeroes & \bI
    \end{pmatrix}
   \left( \bHh^{-1} + o(1) \right)     \begin{pmatrix}
                                                \zeroes & \zeroes \\
                                                \zeroes & \bI
                                              \end{pmatrix}
 + o(1)
\end{equation}
Here multiplication to the left and right by the matrix
$ \begin{pmatrix}
    \zeroes & \zeroes \\
    \zeroes & \bI
  \end{pmatrix} $ selects the blocks with valuation $\leq \frac{s}{2}$. We
  therefore only need to compute
the relevant part in $\bHh^{-1}$, which we can do by block matrix inversion:
\[ 
  \bHh^{-1} =
  \begin{cases}
    \begin{pmatrix}
      * & * \\
      *  & \tau^{-1} \bI
    \end{pmatrix}, & \mathrm{\ if\ } \nu_\ell<\frac{s}{2}, \\
    \begin{pmatrix}
      * & * & * \\
      * &   (\bS_{\ell}+\tau \bI)^{-1} & 0 \\
      *  & 0 & \tau^{-1} \bI
    \end{pmatrix} & \mathrm{\ otherwise}.
  \end{cases}
\]
Finally, inserting into eq. \eqref{eq:reginv-mod-scaling-3} and simplifying, we obtain the following expressions:
\begin{itemize}
\item in case $\nu_\ell< \frac{s}{2}$ (non-matching ) and $\bH_{\le \ell, \le \ell}$ invertible, then for all $\tau$
\begin{equation}\label{eq:reginv-scaling-nonmatching}
\rInvK{s}{\tau}  = 
\begin{pmatrix}
\bI_{b} & \zeroes  \\
\zeroes & \zeroes
\end{pmatrix}
+ o(1),
\end{equation}
where  $b = \sum_{i=0}^{\ell} b_i$.

\item in case $\nu_\ell = \frac{s}{2}$ (matching) and $\bH_{<\ell, <\ell}$ invertible (so that the Schur complement $\bS_\ell$ is well defined), then
\begin{equation}\label{eq:reginv-scaling-matching}
\rInvK{s}{\tau}  = 
 \begin{pmatrix}
\bI_{b-b_\ell} & \zeroes  &\zeroes \\
\zeroes & \bS_\ell(\bS_\ell+\tau\bI)^{-1} & \zeroes \\
\zeroes & \zeroes &\zeroes 
\end{pmatrix}
+ o(1),
\end{equation}
for all $\tau$ such that the inverse exists (e.g., for all positive $\tau$ in the SPD case).
\end{itemize}
The equations  \eqref{eq:reginv-scaling-nonmatching} and \eqref{eq:reginv-scaling-matching} show that (for $\Kb$ given in \eqref{eq:ASE-H} or \eqref{eq:ASE-H-full})
\[
\lim_{\varepsilon \to 0} \rInvK{s}{\tau} =   \bM_{s,\tau}(\Kb) 
\]
where $s \le \frac{\nu_j}{2}$ and $j$ is defined in the statement of the theorem
(thus $\ell \le j$ in the above equations).

Now, invoking \Cref{prop:RI-asymp-equiv}, we can already conclude the following:
\begin{itemize}
\item If $\bH_{\le
    j, \le j}$, for $j \le p$  is invertible, then $\bKe \equivpar{2\nu_j} \bC(\varepsilon)$, with $\bC$ given by the right-hand-side of
  \eqref{eq:ASE-H}
\item Further, if $\bH$ is invertible,   $\bKe \equivinf \bC(\varepsilon)$, with $\bC$ given by the right-hand-side of
  \eqref{eq:ASE-H-full}. 
\end{itemize}

To finish the proof it is enough to check that eq. \eqref{eq:ASE-H} or
\eqref{eq:ASE-H-full} indeed describe (truncated) ASEs. Here we invoke 
\Cref{prop:how-to-recognise-an-ASE}. It is easy to verify that the matrices in
eqs. \eqref{eq:ASE-H} and \eqref{eq:ASE-H} can be rotated into diagonal matrices
with homogeneous entries, and that these equations indeed described the
(truncated) ASE (which is unique by  \Cref{prop:trunc-pse-is-proj}).
 \end{proof}

\section{Generalised kernel form}
\label{sec:generalised-kernel-matrices}

The results of the previous section are only directly useful if the matrix has
entries with different orders of magnitude in $\varepsilon$, so that a non-trivial
scaling matrix can be used. An example of a matrix that only has the trivial
scaling are kernel matrices. Recall (eq. \eqref{eq:kernel-example-rbf})  that
kernel matrices have the following form:
\[ \bKe = \left[ \psi\left(\varepsilon\norm{\bx_i-\bx_j}\right)
  \right]_{i,j=1}^n \] for some function $\psi$. In particular, the popular
``Gaussian'' kernel matrix has $\psi(r) = \exp(-r^2)$.

In the case of the Gaussian kernel, inserting the expansion $\exp(-r^2) =
1-r^2 +\frac{1}{2}r^4 - \dots $, the kernel matrix can be expanded as a series
in $\ep2$, where each term is a matrix of distances raised to some power:
\[ \bK_{2l} = \frac{1}{l!} \left[ (x_i-y_i)^{2l} \right]_{i=1,j=1}^n  \]
Using the binomial theorem, we can expand the distances in terms of monomials:
\[ \bK_{2l} = \frac{1}{l!} \left[ \sum_{q=0}^{2l} \binom{2l}{q} (-1)^q
    x_i^{q}y_i^{2l-q} \right]_{i=1,j=1}^n  \]
If we note $\bv_i = [x_j^i]_{j=1}^n$, this results in the expansion:
\[ \bKe = \bv_0\bv_0^\T - \ep2 \left( \bv_2\bv_0^\T-2\bv_1\bv_1^\T +
    \bv_0\bv_2^\T \right) + \ldots\]
Recall that $\bv_0 = \ones$, the constant vector, so that every entry in $\bKe$
is $\O(1)$ (more precisely, has zero valuation). We can only use the trivial
scaling matrix, and so \Cref{thm:ASE-H} only tells us about the
eigenvalues and eigenvectors of valuation $0$. There is only one such
eigenvalue, since $\bK_0=\bv_0\bv_0^\T$ is of rank one. 

Thus, \Cref{thm:ASE-H} is not powerful enough to directly characterise
the ASE of all analytic perturbations. Better results are needed, and this
section we will  consider a generalisation of \eqref{eq:H-scaling}, and characterise the ASE for matrices of the form:
\begin{equation}
  \label{eq:generalised-kernel}
  \bKe = \bV \bDe \left( \bW + o(1) \right)\bDe \bV^\T
\end{equation}
Although this form may seem abstract, it includes very general kernel
matrices in the flat limit. We will begin by defining these matrices in more
detail, explaining what $\bV$ and $\bW$ correspond to in eq.
\eqref{eq:generalised-kernel}, and then derive the ASE of matrices in this form. 

\subsection{ASE of matrices in generalised kernel form}
\label{sec:ase-generalised-kernels}

We need to make the notation more precise. Again, we define the scaling matrix 
similarly to \eqref{eq:Delta-scaling}
\begin{equation}\label{eq:Delta_a}
\bDe = \matr{\Delta_a} (\varepsilon)=\diag( \underbrace{\varepsilon^{\nu_0},\dots, \varepsilon^{\nu_0}}_{a_0}, \underbrace{\varepsilon^{\nu_1}, \dots
  \varepsilon^{\nu_1}}_{a_1}, \ldots,\underbrace{\varepsilon^{\nu_p} \dots \varepsilon^{\nu_p}}_{a_p}),
\end{equation}
where  each valuation $\nu_i$ is repeated $a_i$ times and $\nu_0 < \nu_1 < \cdots < \nu_p$.

We assume that $\bV \in \RR^{n\times \sum a_i}$  and is partitioned according to the valuations, as
\begin{equation}\label{eq:V}
\bV =  \begin{pmatrix}  \bV_0 & \bV_1 & \ldots & \bV_p \end{pmatrix},
\end{equation}
so that $\bV_i \in \RR^{n \times a_i}$.
We will use the QR factorisation of $\bV$, which we arrange in the block-upper triangular form:
\begin{equation}
  \label{eq:block-QR}
  \bV = \bQ \bR = 
  \begin{pmatrix}
    \bQ_0 & \bQ_1 & \ldots & \bQ_p
  \end{pmatrix}
  \begin{pmatrix}
    \bR_{0,0} & \bR_{0,1} & \ldots &  \ldots & \bR_{0,p} \\
              & \bR_{1,1} & \bR_{1,2}  & \ldots & \bR_{1,p} \\
              & & \bR_{2,2} & \ldots & \bR_{2,p} \\
              & & & \ddots & \vdots \\
              & & &            & \bR_{p,p}
  \end{pmatrix},
\end{equation}
so that the blocks $\bR_{i,j} \in \RR^{b_i \times a_j}$ and $\matr{Q}_r \in \RR^{n \times b_i}$.

\begin{remark}\label{rem:assumptions_V}
Under the assumption that $\rank \bV = n$, the numbers $b_i$ sum to $n$, and $\bQ \in \RR^{n \times n}$ is a square matrix. Moreover $b_i$ measures the new dimensions introduced by $\bV_i$, i.e.,
\[
b_i = 
\begin{cases}
\rank \bV_0, & i = 0, \\
\rank \bV_{\le i} -\rank \bV_{< i} , & i > 0, \\
\end{cases}
\]
where we assume that $b_i > 0$ (i.e. $\rank \bV_{\le i} > \rank \bV_{< i}$).
\end{remark}
\color{black}

The matrix $\bW  \in \RR^{\sum a_i \times \sum a_i}$ in eq.
\eqref{eq:generalised-kernel} is also divided into blocks
\[
  \bW =
  \begin{pmatrix}
    \bW_{0,0} & \bW_{0,1} & \dots & \bW_{0,p} \\
    \bW_{1,0} & \bW_{1,1} & \dots & \bW_{1,p}\\
    \vdots & \vdots & \dots & \dots \\
    \bW_{p,0} & \bW_{p,1} & \dots & \bW_{p,p}
  \end{pmatrix}.
\]
according to the structure of the valuations in \eqref{eq:Delta_a}.

Finally, we define the following matrix:
\begin{equation}
  \label{eq:H-triangular-asymptotics}
  \bH =
  \begin{pmatrix}
    \bR_{0,0} & & & & \\
              & \bR_{1,1} & & & \\
              & & & \ddots &  \\
              & & & & \bR_{p,p}
  \end{pmatrix}
  \bW
  \begin{pmatrix}
    \bR_{0,0}^\T & & & & \\
                & \bR_{1,1}^\T & & & \\
                & & & \ddots &  \\
                & & & & \bR_{p,p}^\T
  \end{pmatrix}
\end{equation}
The sequence of Schur complements in $\bH$ are defined in the same way as for
\Cref{thm:ASE-H}, see \eqref{eq:ScH}.

With this notation, we obtain the following generalisation of \Cref{thm:ASE-H}:
\begin{theorem}
  \label{thm:ASE-generalised-kernel-form}
Let $\bKe$ be given in \eqref {eq:generalised-kernel}, with $\bW$ invertible,  $\bH$ be as in \eqref{eq:triangular-asymptotics},
and $\bV$ satisfying assumptions of \Cref{rem:assumptions_V}. Let  $\bS_0, \ldots, \bS_p$ denote the Schur complements in $\bH$.
\begin{enumerate}
\item Then the ASE is given by
  \begin{equation}
    \label{eq:kbar-generalised-kernel}
    \Kbe = \sum_{i=0}^p  \varepsilon^{2 \nu_i} \bQ_i \bS_i \bQ_i^\T.
  \end{equation}
\item We have $\bS_0 = \bW_{0,0}$, and for any $j>0$ such that $\bV_{\le j-1}$ is full  column rank, $\bS_j$  admits a simpler expression via Schur complements of $\bW$
  \begin{equation}
    \label{eq:schur-compl-H}
    \bS_j=\bR_{j,j} \left( \bW_{j,j} - \bW_{j,<j}(\bW_{< j, <j})^{-1}\bW_{< j,j}  \right)\bR_{j,j}^\T,
  \end{equation}
where 	 the matrices $\bW_{j,< j}$ and $\bW_{<j,j} =\bW_{j,< j}^{\T}$ are defined as
\[
\bW_{j,<j} = \begin{pmatrix}\bW_{j,0}  & \bW_{j,1}  & \cdots & \bW_{j,j-1} \end{pmatrix}.
\]
\end{enumerate}
\end{theorem}
The proof is deferred to the next subsection. 
\begin{example}[\Cref{ex:3x3} revisited]
  \label{ex:3x3-revisited}
  \Cref{ex:3x3} is a case where \Cref{thm:ASE-H} fails to characterise
  all eigenvalues. We show that \Cref{thm:ASE-generalised-kernel-form}
  succeeds.
  In this example we have:
  \[ \bKe =
    \begin{pmatrix}
      1 & \varepsilon & \varepsilon \\
      \varepsilon & \varepsilon^3 & -\varepsilon^3 \\
      \varepsilon & -\varepsilon^3 & \varepsilon^3 \\
    \end{pmatrix}
  \]
  One can check that $\bKe$ can also be written as
  \[
    \bKe = \bV
    \begin{pmatrix}
      1 & &\\ 
        & \varepsilon &\\
      & & \varepsilon^{\frac{3}{2}}
    \end{pmatrix}
    \begin{pmatrix}
      1 & 1& 0\\ 
       1 &0 & 0\\
       0 &0 & 1
    \end{pmatrix}
    \begin{pmatrix}
      1 & &\\ 
        & \varepsilon &\\
        & & \varepsilon^{\frac{3}{2}}
    \end{pmatrix}
  \]
  with
  \[
    \bV =
    \begin{pmatrix}
      1 & 0 & 0 \\
      0 & 1 & -1 \\
      0 & 1 & 1
    \end{pmatrix}.
  \]
  This is a form compatible with \Cref{thm:ASE-generalised-kernel-form}.
  Since $\bV$ already has orthogonal columns we can write
  \[ \bV = \bQ \bR =     \left( \begin{array}{c|c|c}
    1 & 0 & 0 \\
    0 & \frac{1}{\sqrt{2}} & -\frac{1}{\sqrt{2}} \\
    0 & \frac{1}{\sqrt{2}} & \frac{1}{\sqrt{2}}
  \end{array} \right)
\left( 
  \begin{array}{c|c|c}
    1 & & \\
      & \sqrt{2} & \\
    & & \sqrt{2}
  \end{array} \right)
  .
\]
We highlight the block structure corresponding to the successive valuations $0,1,\frac{3}{2}$.
In the notation of this section, we have (eq. (\ref{eq:H-triangular-asymptotics})):
\[
  \bH =  \left(      \setstretch{1.25}
 \begin{array}{c|c|c}
    1 & & \\
    \hline
      & \sqrt{2} & \\
    \hline
      & & \sqrt{2}
  \end{array} \right)
\left( \setstretch{1.25}
  \begin{array}{c|c|c}
    1 & 1  & \\
    \hline
    1  &  & \\
    \hline
      & & 1
  \end{array}  \right)
\left( \setstretch{1.25}
\begin{array}{c|c|c}
  1 & & \\
  \hline
    & \sqrt{2} & \\
  \hline
    & & \sqrt{2}
\end{array} \right)
=
\begin{pmatrix}
  1 & \sqrt{2} & \\
  \sqrt{2} & & \\
  & & 2
\end{pmatrix}
\]
The Schur complements in $\bH$ are $\bS_0=1, \bS_1= -2, \bS_2=2$. Applying
\Cref{thm:ASE-generalised-kernel-form}, the ASE of $\bK$ equals:
\[ \Kbe = \bq_0 \bq_0^\T - 2 \varepsilon^2  \bq_1 \bq_1^\T + 2 \varepsilon^3  \bq_2
  \bq_2^\T =     \begin{pmatrix}
    1 & & \\
      & & \\
      & & \\
  \end{pmatrix}
  - \varepsilon^2
  \begin{pmatrix}
    & & \\
    & 1 & 1 \\
    & 1 & 1\\
  \end{pmatrix}
  + \varepsilon^3
  \begin{pmatrix}
    & & \\
    & 1 & -1 \\
    & -1 & 1\\
  \end{pmatrix}
 \]
 Compared to the previous attempt, we have managed to identify the eigenvalue of
 order $\O(\varepsilon^3)$. 
\end{example}

\subsection{Proof of \Cref{thm:ASE-generalised-kernel-form}}
\label{sec:proof-of-theorem-ASE-kernel-form}

We begin with a lemma that allows us to convert the form of eq. (\ref{eq:generalised-kernel})
to the simpler form we used previously in section \ref{sec:diagonal-scalings},
that of eq. (\ref{eq:H-scaling}). The lemma already appears in a different form
in \cite{BarthelmeUsevich:KernelsFlatLimit}.
\begin{lemma}
  \label{lem:triangular-asymptotics}
Let $\bV$, $\bQ$, $\bR$ be as in \eqref{eq:V},  \eqref{eq:block-QR}, and $\bH$
as in \eqref{eq:H-triangular-asymptotics}. 
Then we have the following asymptotic equivalence:
  \begin{equation}
    \label{eq:triangular-asymptotics}
    \bR \matr{\Delta_a}(\varepsilon) (\bW + o(1))\matr{\Delta_a}(\varepsilon) \bR^\T = \matr{\Delta_b}(\varepsilon) (\bH + o(1)) \matr{\Delta_b}(\varepsilon)
  \end{equation}
  where $\matr{\Delta_a}(\varepsilon)$ and $\matr{\Delta_b}(\varepsilon)$ are defined in \eqref{eq:Delta_a} and \eqref{eq:Delta-scaling}, respectively, and
\end{lemma}
\begin{proof}
  We denote $\bA(\varepsilon) = \bR \bDe (\bW + o(1))\bDe \bR^\T$, and $\bA_{i,j}$
  the block $i,j$ in the partitioning induced by the valuations in the scaling
  matrix $\bDe$.
  \begin{align*}
    \bA_{i,j}(\varepsilon) &= \sum_{k\geq i,l \geq j} \bR_{i,k}\varepsilon^{\nu_i}
                          (\bW_{k,l}+o(1)) (\bR)^\T_{l,j} \varepsilon^{\nu_j} \\
                        &= \varepsilon^{\nu_i} (\bR_{i,i}
                          \bW_{i,i} (\bR_{i,i})^\T + o(1)) \varepsilon^{\nu_j} 
  \end{align*}
  since the valuations in $\bDe$ are increasing. Eq. 
  (\ref{eq:H-triangular-asymptotics}) is exactly this result expressed for all
  blocks. 
\end{proof}

Given this lemma, the proof of \Cref{thm:ASE-generalised-kernel-form} is straightforward.
\begin{proof}[Proof of \Cref{thm:ASE-generalised-kernel-form}]
For the first part of the theorem (statement (1)), consider the matrix
\[ \bK'(\varepsilon) = \bQ^\T \bKe \bQ = \bR \bDe (\bH+o(1))^{-1} \bDe \bR^\T \]
and note that by \Cref{lem:properties-ase}, $\underline{\bK'(\varepsilon)} =
\bQ^\T \Kbe  \bQ$. \Cref{lem:triangular-asymptotics} lets us apply theorem
\Cref{thm:ASE-H} to $\bK'$. We obtain eq.
(\ref{eq:kbar-generalised-kernel}) via $\Kbe =
\bQ \underline{\bK'(\varepsilon)} \bQ^\T $.

For the second part (statement (2)), note that in this case since $\rank \bR = \rank \bV = n$, all the diagonal blocks $\bR_{i,i}$  are full row rank.
This implies that $\bH$ is invertible (thanks to invertibility of $\bW$). The simplified expression for the
Schur complement $\bS_i$ (eq. \eqref{eq:schur-compl-H}) can be obtained as
follows. If  $\bV_{\le j-1}$ is full  column rank, then $\bR_{i,i}$ are square for $j <i$ and the matrix
\[
\bA = \diag (\bS_0,\ldots, \bS_{j-1})
\]
is invertible.
Then the Schur complement is obtained as
\begin{align*}
\bS_j & = \bR_{j,j} \bW_{j,j} \bR_{j,j}^{\T}-  \bR_{j,j} \bW_{j,<j}\bA^{\T} (\bA \bW_{<j,<j} \bA^{\T})^{-1}  \bA \bW_{<j,j}\bR_{j,j}^{\T}\\
& = 
\bR_{j,j} \left( \bW_{j,j} - \bW_{j,< j}(\bW_{< j, < j})^{-1}\bW_{<j,j}  \right)\bR_{j,j}^\T,
\end{align*}
which completes the proof.
\end{proof}
\section{Application to kernel matrices}
\label{sec:kernel-matrices}

In this section, we apply \Cref{thm:ASE-generalised-kernel-form} to
kernel matrices, to settle a conjecture from
\cite{BarthelmeUsevich:KernelsFlatLimit}. The only difficulty is to show that
kernel matrices can indeed be written in the form required by theorem
\ref{thm:ASE-generalised-kernel-form}, and explicitate the matrices involved.

\subsection{Kernel matrices: background and notation}

We need to briefly recall some definitions and notation on kernel matrices. For
more information on these matrices, we refer the reader to
\cite{schaback2006kernel,wendland2004book,fasshauer2007meshfree}. We follow the
notation used in \cite{BarthelmeUsevich:KernelsFlatLimit}.

Kernel matrices are formed from a set of $n$ points in $\R^d$, noted $\X = \{
\bx_1, \bx_2,\dots, \bx_n\}$. We define
\begin{equation}
  \label{eq:kernel-matrix-gen}
  \bKe = \left[ k(\varepsilon\bx_i,\varepsilon\bx_j) \right]_{i,j=1}^{n,n}
\end{equation}
where $k(\bx,\by)$ is a positive-definite \emph{kernel function} and $\varepsilon$
is a (spatial) non-negative scaling parameter. We seek to characterise the ASE of kernel
matrices in the flat limit $\flatlim$.

The class of kernel functions is very wide, but the most commonly-used are
\emph{radial}, meaning that $k(\bx,\by)$ only depends on the (Euclidean)
distance $\norm{\bx-\by}$:
\begin{equation}
  \label{eq:def-radial}
  k(\bx,\by) = \psi(\norm{\bx-\by})
\end{equation}
Radial kernels are particularly easy to work with because the flat limit
expansion of $k(\varepsilon \bx,\varepsilon \by)$ can be obtained from an expansion of
$\psi(s)$ at 0, i.e.:
\begin{equation}
  \label{eq:radial-expansion}
  k(\varepsilon\bx,\varepsilon\by) = \psi(\varepsilon\norm{\bx-\by}) = \psi_0+\varepsilon \psi_1 \norm{\bx-\by}  +\ep2 \psi_2 \norm{\bx-\by}^2 + \ep3 \psi_3 \norm{\bx-\by}^3 + \dots
\end{equation}
In this section we assume that $\psi(s)$ is analytic at 0, so that we are
dealing with analytic perturbations.

\begin{remark}
While we consider radial kernels in this section, the result will also hold for other smooth kernels (as long as the kernel matrices analytic in the scaling parameter $\varepsilon$), similarly to \cite{BarthelmeUsevich:KernelsFlatLimit}.
\end{remark}

The following criterion, called the ``regularity index'', is key for
characterising the flat limit of kernel matrices:
\begin{definition}[Regularity index]
  Let $k(\bx,\by) = \psi(\norm{\bx - \by})$ a radial kernel, and $\psi(s)$ have the following expansion at $0$:
  \begin{equation}
    \label{eq:regularity-index}
    \psi(s) = \psi_0 + \psi_1 s + \psi_2 s^2 + \ldots
  \end{equation}
  We say $k$ has regularity index $r$ if $\psi_{2r-1} \neq 0$, and $\psi_{2c-1}=0$ for
  $c<r$, i.e. $\psi_{2r-1}$ is the first non-zero odd term in the expansion.
  Kernels with $r<\infty$ are said to be ``finitely smooth'', kernels with
  $r=\infty$ are said to be ``completely smooth''.
\end{definition}
The spectral behaviour of kernel matrices depends in the most part on the
regularity index, although we cannot provide a concise explanation of why this
is the case (see \cite{barthelme2022gaussian} for a discussion).

Two kernels with widely different flat limit behaviour are the Gaussian kernel
and the exponential kernel, with $r=\infty$ and $r=1$, respectively:
\begin{example}[Examples of kernels with different regularity coefficients]
  The Gaussian kernel (eq. \eqref{eq:kernel-example}) corresponds to
  $\psi(s)=\exp(-s^2)=1-s^2+\frac{1}{2}s^4-\frac{1}{6}s^6+\dots$,
  which contains only \emph{even} monomials in $s$ ($s^0,s^2,s^4,\dots$). This
  leads to a small-$\varepsilon$ expansion that contains only even powers of $\varepsilon$, so that
  the Gaussian kernel has therefore regularity $r=\infty$. 

  Contrast this to the so-called \emph{exponential} kernel, which has
  \begin{equation}
    \label{eq:exponential-kernel}
    \psi(s) = \exp(-s) = 1-s+\frac{1}{2}s^2-\frac{1}{6}s^3+\dots
  \end{equation}
  where the first odd power of $s$ is $s^1$. The regularity of the exponential
  kernel is therefore $r=1$.
  
  The following kernel is a special case of the Matérn family of kernels \cite{stein1999interpolation} with $r=2$:
  \begin{equation}
    \label{eq:matern-r2}
    \psi(s) = (1+s)\exp(-s) = (1+s)(1-s+\frac{1}{2}s^2-\frac{1}{6}s^3\dots) = 1-\frac{s^2}{2}  + \frac{s^3}{3} + \dots
  \end{equation}
  Compared to the exponential kernel, the term in $s$ drops out, but the term in
  $s^3$ remains, which increases $r$ from 1 to 2. 
\end{example}

\subsection{Vandermonde, Wronskian and distance matrices}
To express kernel matrices in the requisite form we need some standard notation for
multivariate polynomials (see \cite{BarthelmeUsevich:KernelsFlatLimit} for
details). 

Let $\vect{x} =
\begin{pmatrix}
  x_1 & x_2 & \ldots & x_d
\end{pmatrix}^\T
\in \R^d $. A monomial in $\vect{x}$ is a function of
the form:
\[ \vect{x}^{\vect{\alpha}} = \prod_{i=1}^d x_i^{\alpha_i}  \]
for $\vect{\alpha} \in \mathbb{N}^d$ (a multi-index). The degree of a monomial
is defined $|\vect{\alpha}|=\sum_{i=1}^d \alpha_i$. For instance:
$ \vect{x}^{(1,3,1)} = x_1^1x_2^3x_3^1$  has degree 5, so does $
\vect{x}^{(2,2,1)} = x_1^2x_2^2x_3^1$. The numbers of monomials of degree $\leq
s$  and degree $=s$ in dimension $d$ are given by 
\begin{equation}
  \label{eq:dim-polynomials}
  \PP_{s,d} =  {k+d \choose d},\quad
		  \HH_{s,d} =  {k+d-1 \choose d-1}
\end{equation}
respectively.
By 
evaluating monomials of degree $\leq s$ on a discrete set of nodes $\X = \{
\bx_1, \dots, \bx_n \}$, we form a matrix:
\begin{definition}[Vandermonde matrix]
  The Vandermonde matrix of degree $\leq s$ on nodes $\X = \{
  \bx_1, \dots, \bx_n \}$
  \begin{equation}
    \label{eq:vdm-matrix}
    \bV_{\leq s} = \left[ \bx_{i}^{\bal} \right]_{i \in \{1,\dots,n\},|\bal|\leq s }
  \end{equation}
\end{definition}
The nodes vary along the rows, the monomials (indexed by $\bal$) along the
columns. Which (degree-graded) monomial order is used is irrelevant for
our results. 
\begin{example}
In dimension one, the generalised Vandermonde matrix simplifies to
the classical Vandermonde matrix:
\begin{equation}
  \label{eq:vdm-matrix-1d}
  \bV_{\leq s} = \left[ x_{i}^{j} \right]_{i \in \{1,\dots,n\},j \in \{0,\dots,s\} }
\end{equation}
\end{example}

By eq. \eqref{eq:dim-polynomials}, the matrix $\bV_{\leq s}$ has dimension $n
\times  \PP_{s,d} = n \times \binom{s+d}{d}$ in dimension $d$. One needs to keep
in mind two important differences between the univariate and the multivariate
case:
\begin{enumerate}
\item In the univariate case, the matrix $\bV_{\leq s}$ always has full column
  rank, if $s \leq n-1$ and if the nodes in $\X$ are distinct. This is no longer
  the case in the multivariate case. For instance, if the nodes lie on a line, then
  $\bV_{\leq 1}$ has rank 2 instead of rank $1+d$. When the matrix $\bV_{ \leq
    s}$ is rank-deficient, we say the nodes are \emph{non-unisolvent} at degree $s$.
\item In the univariate case, by picking $s=n-1$ we obtain a square matrix
  $\bV_{\leq n-1}$. In the multivariate case this may or may not be possible
  depending on $n$ and $d$. For instance, in dimension $d=2$, the size of
  $\bV_{\leq s}$ is $1,3,6,10,\dots$ with $s=0,1,2,3,\dots$. If $n=5$ then
  $\bV_{\leq 1}$ is too narrow and $\bV_{\leq 2}$ too wide. 
\end{enumerate}
We also split the Vandermonde matrices in blocks, which correspond to fixed degrees of monomials:
\[
\bV_{\le s} = \begin{pmatrix}\bV_{0} & \bV_{1} & \cdots & \bV_{s} \end{pmatrix}, \quad \bV_{s} = \left[ \bx_{i}^{\bal} \right]_{i \in \{1,\dots,n\},|\bal|= s }
\]
The particular order of monomials of the same degree will not be important for what follows, but we assume it fixed. 
\begin{example}
In dimension $2$, the generalised Vandermonde matrix  has blocks $  \bV_{s}$ with $1,2,3,\ldots$ columns:
\begin{equation}
  \label{eq:vdm-matrix-2d}
 \matr{V}_{\le 2} =
\left[\begin{array}{c|cc|ccc}
1 & y_1 & z_1 & y_1^2 & y_1z_1 & z_1^2 \\
1 & y_2 & z_2 & y_2^2 & y_2z_2 & z_2^2 \\
\vdots & \vdots & \vdots  & \vdots  & \vdots  & \vdots  \\
1 & y_n & z_n & y_n^2 & y_nz_n & z_n^2 \\
\end{array}\right],
\end{equation}
where 
  \[
\X = \{\left[\begin{smallmatrix}y_1 \\z_1 \end{smallmatrix}\right], \left[\begin{smallmatrix}y_2 \\z_2 \end{smallmatrix}\right], \ldots, \left[\begin{smallmatrix}y_n \\z_n \end{smallmatrix}\right]\},
\]
with a particular chosen ordering of monomials.
\end{example}

Alongside Vandermonde matrices, we need to define Wronskian matrices, which are
matrices of partial derivatives of the kernel. The full Wronskian matrix is an
infinite matrix that contains partial derivatives at all orders. Here we only
need submatrices containing derivatives of finite order:
\begin{definition}\label{def:wronskian}
  Let $k$ a radial kernel with regularity index $r$, and let $d$ be fixed.
  Then for $i,j \ge 0$ such that $i+j \le 2r -2$ we define the Wronskian submatrix $\bW_{i,j} \in \R^{\HH_{i,d} \times \HH_{j,d}}$ as
  \[
\bW_{i,j} = \left[ \frac{\left.\frac{\partial^{q}}{\partial \bx^{\bal} \partial \by^{\bbet}}  k(\bx,\by) \right|_{\bx,\by = 0}}{\bal!\bbet!} \right]_{|\bal| = i, |\bbet| = j},   
  \]
where the columns and the rows and columns of $\bW_{i,j}$ are indexed by the
multi-indices whose order is consistent with ordering the degrees in the
Vandermonde matrices \eqref{eq:vdm-matrix} (i.e., rows and columns of
$\bW_{i,j}$ are in the same order as columns for $\bV_i$ and $\bV_j$
respectively). 
\end{definition}
See the appendix of \cite{barthelme2022gaussian} for convenient formulas for
Wronskians of common kernels.

We can also define stacked Wronskian matrices $\bW_{\le i, \le i}$, for $i < r$ as
\begin{equation}\label{eq:wronskian_stacked}
  \bW_{\le i, \le i}  =
  \begin{pmatrix}
    \bW_{(0,0)} & \bW_{(0,1)} & \cdots &  \bW_{(0,i)}\\
    \bW_{(1,0)} & \bW_{(1,1)} & \cdots & \bW_{(1,i)} \\
    \vdots           & \vdots           &          & \vdots \\
    \bW_{(i,0)}  & \bW_{(i,1)}  & \cdots & \bW_{(i,i)} 
  \end{pmatrix}
\end{equation}
This matrix is of size ${\PP_{i,d} \times \PP_{i,d}}$ and contain all partial derivatives up to order $2i$.
Moreover, it possesses the following nice property:
\begin{lemma}\label{lem:spd_wronskian}
  For a (strictly) positive definite kernel, all the Wronskian matrices are (strictly) positive definite.
\end{lemma}
\begin{proof}
  See \cite{barthelme2022gaussian}.
\end{proof}

Finally, the following matrices are required in the expansion of finitely-smooth
kernels:
\begin{equation}
  \label{eq:distance-matrices}
  \bD^{(q)} = \left[ \norm{\bx_i - \bx_j}^q \right]_{i,j=1}^n.
\end{equation}
We conclude this subsection by a result on conditional positive definiteness of
$\bD^{(q)}$ for odd order $d$.
\begin{lemma}\label{lem:distance_cpd}
If $q = 2r-1$ for integer $r \geq 1$,  and $\bV_{\le r-1}$ is full column rank, then the matrix
\[
(-1)^r \bA^\T \bD^{(q)} \bA
\]
is strictly positive definite for any full column rank matrix $\bA$
with $\bA^\T \bV_{\le r-1} = 0$.
\end{lemma}
\begin{proof}
  See \cite{fasshauer2007meshfree}, ch. 8.
\end{proof}

We can now show that kernel matrices have the form required to apply theorem
\ref{thm:ASE-generalised-kernel-form}. We separate the completely smooth and
finitely smooth cases.

\subsection{Results in for smooth radial kernels}
We first analyse the smooth case, where the kernel function is differentiable sufficiently many times.
With some abuse of notation, denote by
\begin{equation}\label{eq:Delta_k}
\matr{\Delta}_k = \matr{\Delta}_k(\varepsilon)   = 
\begin{bmatrix}
1 & & & \\
& \varepsilon I_{ \HH_1} & & \\
&  & \ddots & \\
&  &  & \varepsilon^{k}I_{ \HH_k} 
\end{bmatrix},
\end{equation}
which is a particular case of the matrix \eqref{eq:Delta_a} (corresponding to the choice of  with $a_j = \HH_{j,d}$ and $\nu_j = j$, i.e  the number of repetitions of $\varepsilon^j$ is according to the number of homogeneous polynomials of degree $d$).

\begin{lemma}\label{prop:asymptotic-kernel-smooth}
Let $k(\bx,\by)$ be a positive definite kernel function,
such that $k(\bx,\by)$  is a radial kernel with regularity $r$, and 
 $\X$ be a node set such that $\rank \bV_{\le p-1} = n$ for $p \le r$.

Then the kernel matrix has the following  asymptotic form in $\flatlim$:
  \begin{equation}
    \label{eq:kernel-asymptotic-form-smooth}
    \bKe = \bV_{\le p-1} \matr{\Delta}_{p-1}(\varepsilon)  (\bW_{\le p-1} + o(1)) \matr{\Delta}_{p-1}(\varepsilon)  \bV^\T_{\le p-1}.
  \end{equation}
where $\bV = \bV_{\le p}$ is the Vandermonde matrix described in \eqref{eq:vdm-matrix}, $\bW = \bW_{\le p,\le p}$ is the Wronskian matrix from  \eqref{eq:wronskian_stacked}, and   $\bDe = \matr{\Delta}_k$ is the diagonal matrix  defined in \eqref{eq:Delta_k}.
\end{lemma}
\begin{proof}
The proof is given in \cref{sec:appendix}.
\end{proof}

Note that the same expansion holds for general smooth kernels (not necessarily radial).
\begin{lemma}\label{prop:asymptotic-kernel-smooth-general}
$k(\bx,\by)$ be a positive definite kernel function with $k \in \mathcal{C}^{p,p}$ and 
 $\X$ be a node set such that $\rank \bV_{\le p-1} = n$.
Then $\bKe$ has the same expansion as in \eqref{eq:kernel-asymptotic-form-smooth}
\end{lemma}
\begin{proof}
The proof is given in \cref{sec:appendix}.
\end{proof}

\Cref{prop:asymptotic-kernel-smooth} and \Cref{prop:asymptotic-kernel-smooth-general} help us to establish find the ASE for the completely smooth case.

\begin{theorem}\label{thm:ase-smooth}
Let $\bKe$ satisfy the conditions of \Cref{prop:asymptotic-kernel-smooth} or \Cref{prop:asymptotic-kernel-smooth-general}, and, in addition $p$ is the smallest such number (i.e., $\rank \bV_{\le p-1} < n$ but $\rank \bV_{\le p} = n$).
Let $\bQ_0,\ldots,\bQ_p$ and $\bR_{0,0}, \ldots, \bR_{p,p}$ come from the block QR factorization of the matrix
\[
\bV = \bV_{\le p} = \begin{pmatrix}\bV_{0} & \bV_{1} & \cdots & \bV_{p} \end{pmatrix}.
\]
Then  the ASE of $\bKe$ is given by
\begin{equation}
\label{eq:kbar-completely smooth}
\Kbe = \sum_{i=0}^p \bQ_i \bS_i \bQ_i^\T \varepsilon^{2 p},
\end{equation}
where $\bS_i$ are the Schur complements of the block matrix $\bH$ defined in \eqref{eq:H-triangular-asymptotics}.
\end{theorem}
\begin{proof}
By \Cref{prop:asymptotic-kernel-smooth}, $\bKe$ has factorization \eqref{eq:kernel-asymptotic-form-smooth}, where the matrix $\bW$ is strictly positive definite by \Cref{lem:spd_wronskian}. 
Therefore, we can apply \Cref{thm:ASE-generalised-kernel-form}, and the statements  follow from \Cref{thm:ASE-generalised-kernel-form}.
\end{proof}

In the case when the Vandermonde matrices are full rank, then we can use the
simplified expression for the Schur complements. These expressions lead to the
following corollary, settling Conjecture 1 in \cite{BarthelmeUsevich:KernelsFlatLimit} in the smooth case.

\begin{corollary}\label{cor:ase-smooth-unisolvent}
Let $j \le p$  be such that the matrix $\bV_{\le j}$ is full-column rank.
Then 
\begin{enumerate}
\item  the number of eigenvalues in the $j$-th block (with the order $\varepsilon^{2j}$)  is exactly equal to  $\HH_{j,d}$ (the number of monomials of degree $j$);
\item the Schur complement $\bS_j$ admits the simplified expression \eqref{eq:schur-compl-H}, and therefore, the leading coefficients of the eigenvalues are given by the eigenvalues of $\bS_j$ and the ``leading eigenvectors'' (in the sense of \Cref{def:ASE}) are given by the eigenvectors of 
\begin{equation}\label{eq:ase-term-unisolvent}
\bQ_{j} \bS_{j} \bQ_{j}^\T = 
\begin{cases}
\bV_{0} \bW_{0,0} \bV_{0}^\T, & j =0, \\
\bQ_{j}\bQ_{j}^\T\bV_{j} \left( \bW_{j,j} - \bW_{j,<j}(\bW_{< j, <j})^{-1}\bW_{< j,j}  \right)\bV_{j}^\T \bQ_{j}\bQ_{j}^\T, & j > 0.\\
\end{cases}
\end{equation}
\end{enumerate}
\end{corollary}
\begin{proof}
This follows from the fact that the full rank property $\bV_{\le j}$ implies that all the diagonal blocks in the QR decomposition ($\bR_{i,i}$ for $i \le j$) are square (i.e. $b_i=a_i$ in \eqref{eq:block-QR}).
Therefore we obtain the simplified expression \eqref{eq:schur-compl-H}.
The rest follows from the fact that $\bQ_{j}^\T\bV_{j} = \bR_{j,j}$.
\end{proof}

Note that \Cref{thm:ase-smooth} applies also in the the non-unisolvent case (the case when the matrices $\bV_{\le j-1}$ may not be of full column rank).
However, the expressions for the Schur complements are more complicated than the
ones in \Cref{cor:ase-smooth-unisolvent}, see  \cref{sec:non-unisolvent}.

\subsection{Finitely smooth case }
The finitely smooth case appears when there is no such $p$ that satisfies the conditions of \Cref{prop:asymptotic-kernel-smooth}.
This happens, for example, if $\rank \bV_{\le r-1}  < n$ for a radial kernel of the regularity index $r$.
This case can be treated with the following proposition

\begin{proposition}\label{prop:asymptotic-kernel-finite-smoothness}
Let $\psi$ be with regularity $r$ and  $\rank \bV_{\le r-1} < n$.
Then the kernel matrix has the following  asymptotic form in $\flatlim$: 
  \begin{equation}
    \label{eq:kernel-asymptotic-form-finite-smoothness}
    \bKe = \bV \bDe (\bW + o(1)) \bDe \bV^\T,
  \end{equation}
\begin{itemize}
  \item $\bV = \left[ \bV_{\leq r-1}\ \bA \right]$ where $\bA \in \RR^{n \times a_j}$,  is an
    arbitrary full column rank matrix such that $\mspan \bV = \R^n$ (which implies (with $c = n-\rank \bV_{\le r-1}$). 
  \item $\bDe$ is a diagonal scaling matrix, with block structure 
\[
\bDe=\begin{pmatrix}
1  &&&&\\
& \bm{\varepsilon}^{1} \bI_{\HH_{1,d}} &&& \\
&&\ddots&& \\
&&& \bm{\varepsilon}^{r-1} \bI_{\HH_{r-1,d}} & \\
&&& & \bm{\varepsilon}^{r-\frac{1}{2}} \bI_{c}  \\
\end{pmatrix},
\]
where each block of integer valuation $t$ is of size $\HH_{i,d}$. The only block with fractional valuation is the
    last one. 
  \item $\bW$ is an extended ``Wronskian'' matrix with
    the following structure: 
    \begin{equation}
      \label{eq:wronskian-blocks-fs}
      \bW =
      \begin{pmatrix}
        \bW_{\le r-1,\le r-1} &  \zeroes\\
     \zeroes & \psi_{2r-1} \bA^{\dag}  \bD^{(2r-1)} (\bA^{\dag})^{\T}
      \end{pmatrix}.
    \end{equation}
   \end{itemize}
\end{proposition}
\begin{proof}
The proof is given in \cref{sec:appendix}.
\end{proof}

\begin{theorem}\label{thm:ase-finite-smoothness}
Let $\bKe$ satisfy the conditions of \Cref{prop:asymptotic-kernel-finite-smoothness}, and, in addition $\rank \bV_{\le r-1} < n$.
Let $\bQ_0,\ldots,\bQ_{r-1}$ and $\bR_{0,0}, \ldots, \bR_{r-1,r-1}$ come from the block QR factorization of the matrix
\[
\bV = \bV_{\le r-1} = \begin{pmatrix}\bV_{0} & \bV_{1} & \cdots & \bV_{r-1} \end{pmatrix}.
\]
Then the ASE of $\bKe$ is given by
\begin{equation}
\label{eq:kbar-finite-smoothness}
\Kbe = \sum_{i=0}^{r-1} \bQ_i \bS_i \bQ_i^\T \varepsilon^{2 i} + \varepsilon^{2 r-1} \psi_{2r-1}\bA \bA^{\dag} \bD^{(2r-1)}  \bA \bA^{\dag} ,
\end{equation}
where the terms $\bS_0,\bS_1, \dots $ have exactly the same form as in \Cref{thm:ase-smooth}.
\end{theorem}
\begin{proof}
By \Cref{prop:asymptotic-kernel-finite-smoothness}, $\bKe$ has factorization \eqref{eq:kernel-asymptotic-form-smooth}, where the matrix $\bW$ is strictly positive definite by \Cref{lem:spd_wronskian}. 
Finally we note that, since $\bA^\T \bV_{\le r-1}$, the QR decomposition \eqref{eq:block-QR} of the matrix $\bV$ in \Cref{prop:asymptotic-kernel-finite-smoothness} is given by
\[
\bV  = 
  \begin{pmatrix}
    \bQ_0 &  \ldots & \bQ_{r-1} & \bA
  \end{pmatrix}
  \begin{pmatrix}
    \bR_{0,0} &  \ldots & \bR_{0,r-1} &  \\
              & \ddots & \vdots       &   \\
              &        & \bR_{p,p}    &  \\
              &        &              & \bI_{c} \\
  \end{pmatrix},
\]
with $c = n - \rank \bV_{\le r-1}$
hence, we can  again apply \Cref{thm:ASE-generalised-kernel-form}, and the statements of the Theorem follows from \Cref{thm:ASE-generalised-kernel-form}, 1--2.
\end{proof}

As a corollary of \Cref{thm:ase-finite-smoothness}, we recover both \cite[Conjecture 1]{BarthelmeUsevich:KernelsFlatLimit} and 
\cite[Theorem 6.3]{BarthelmeUsevich:KernelsFlatLimit}
 in the finite smoothness case.
\begin{corollary}
\begin{enumerate}
\item For $j \le r-1$, if $\bV_{\le j}$ if full column rank, then the number of eigenvalues of the order $\varepsilon^{2j}$  is exactly equal to  $\HH_{j,d}$, and the expressions for the Schur complement are as in \Cref{cor:ase-smooth-unisolvent}.
\item If, in addition $\bV_{\le r-1}$ is full column rank, then there are exactly $n - \PP_{r-1,d}$ eigenvalues of order $\varepsilon^{2r-1}$ corresponding to the ASE term $\psi_{2r-1}\bA \bA^{\dag} \bD^{(2r-1)}  \bA \bA^{\dag}$. 
\end{enumerate}
\end{corollary}
\begin{proof}
The first part of the corollary is proved similarly to \Cref{cor:ase-smooth-unisolvent}.
The second part follows from \Cref{lem:distance_cpd}, which implies that  the $\bA^{\dag} \bD^{(2r-1)}  \bA$ has rank $c$ (and so has the matrix $\bA\bA^{\dag} \bD^{(2r-1)}  \bA\bA^{\dag}$).
\end{proof}

\subsection{Non-unisolvent case}\label{sec:non-unisolvent}
Finally, we make some remarks on the non-unisolvent case, i.e. the case where
the Vandermonde matrix $\bV_{\leq s}$ are rank-deficient for some $s$.
We make first the remark on the ranks on the Vandermonde matrices. 
\begin{lemma}\label{lem:vandermonde-rank-increase}
If all the points  $\bx_{i}$, $i \in \{1,\ldots,n\}$ are distinct, then we have that  $\rank \bV_{\leq s} > \rank \bV_{\leq s-1}$ for all $1 \le s < n$.
\end{lemma}
\begin{proof}
If all the points are distinct, then there exists a vector $\ba \in \RR^{n}$ such that all $t_i  = \ba^\T \bx_{i}$ are distinct.
Note that the univariate Vandermonde matrices $\widetilde{\bV}_{\le s} = [t^{j}_i]^{n,s}_{i=1,j=0}$ are full column rank for $s \le n-1$.
Note that the last column of $\widetilde{\bV}_{\le s}$ lies in  $\mspan{\bV}_{s}$. Hence, for all $1 \le s \le n-1$, $\mspan{\bV}_{s}$ contains at least one vector that does not belong to  $\mspan{\bV}_{\le s-1}$, and therefore $\rank \bV_{\leq s} > \rank \bV_{\leq s-1}$.
\end{proof}

Note that Theorems~\ref{thm:ase-smooth}~and~\ref{thm:ase-finite-smoothness} still apply, and as a special case, we generalize the results in \cite{wathen2015eigenvalues} (see \cite[Theorem 8]{wathen2015eigenvalues} where the number of eigenvalues of given order is provided for analytic kernels).
\begin{corollary}\label{cor:ase-non-unisolvent}
Assume that the conditions of \Cref{thm:ase-smooth}~or \Cref{thm:ase-finite-smoothness} hold and $j \le r-1$, and all the points $\bx_i$ are distinct.
\begin{enumerate}
\item If $\rank \bV_{\le j}$ is rank deficient but $\rank \bV_{\le j-1}$ is full column rank (rank $\PP_{j-1,d}$), then there are $\rank \bV_{\le j} - \PP_{j-1,d}$ eigenvalues of the degree $\varepsilon^{2j}$ and the formula \eqref{eq:ase-term-unisolvent} is still valid for the corresponding term of the ASE.
\item If $\rank \bV_{\le j}$ and $\rank \bV_{\le j-1}$ are both rank deficient (i.e., $\rank \bV_{\le j}<\PP_{j,d}, \rank \bV_{\le j-1} < \PP_{j-1,d}$, then there are $\rank \bV_{\le j} - \rank \bV_{\le j-1}$ eigenvalues of degree  $\varepsilon^{2j}$  and the matrix $\bS_j$ is obtained from the Schur complement (with respect to the last block) of
\[
\bH = 
  \begin{pmatrix}
    \bR_{0,0}\bW_{0,0}\bR_{0,0}^\T &  \dots & \bR_{0,0} \bW_{0,j}\bR_{j,j}^\T \\
    \vdots \dots & \dots \\
    \bR_{j,j}\bW_{j,0} \bR_{0,0}^\T & \dots &\bR_{j,j}\bW_{j,j} \bR_{j,j}^\T
  \end{pmatrix}.
\]
\end{enumerate}
\end{corollary}
\begin{proof}
The statements about the numbers of the eigenvalues follow from \Cref{rem:assumptions_V}, where $b_i$ gives the number of rows in the diagonal block $\bQ_{i,i}$ of the QR decomposition. 
Note that $b_i$ is nonzero by \Cref{lem:vandermonde-rank-increase}.
Finally, the expressions for the Schur complement follow from combining
\eqref{eq:ScH} and \eqref{eq:H-triangular-asymptotics}.
\end{proof}

\begin{remark}
  The Vandermonde matrices $\matr{V}_{\le j}$ are rank-deficient if the points
  $x_1 \dots x_n$ are sampled from an algebraic variety having polynomial equations of degree $\leq j$. For
  example, consider points sampled on a circle (i.e., $x_1^2 + x_2^2 = 1$), or,
  in general a conic section, in $d=2$. Then we have that
\[
  \rank \bV_{\le j} = 2j + 1,
\]
which is smaller than $\PP_{j,2} = \binom{j+2}{2}$ as long as $j \ge 2$.
In this case, \Cref{cor:ase-non-unisolvent} gives the limiting eigenvector and eigenvalues for kernel matrices corresponding to nodes on a circle.

The example of the circle can be generalized to the case when the points $\bx_j$
lie on an algebraic variety. In this case, the ranks of the Vandermonde matrix
are connected to the Hilbert function \cite[Ch. 9]{cox1997ideals} of the
corresponding polynomial ideal (see also \cite{altschuler2023kernel} for
examples of Hilbert functions). Note that for some particular algebraic
varieties (e.g., spheres) it may be more beneficial to use some predefined basis
(e.g., spherical harmonics \cite{fornberg2011stable}) instead of multivariate
monomials.
\end{remark}

\subsection{Numerical illustration}
\label{sec:numerics}

We illustrate our results with two examples in dimension 2, and contrast unisolvent to non-unisolvent sets.

\begin{figure}[ht]
  \centering
 
  \includegraphics[width=12cm]{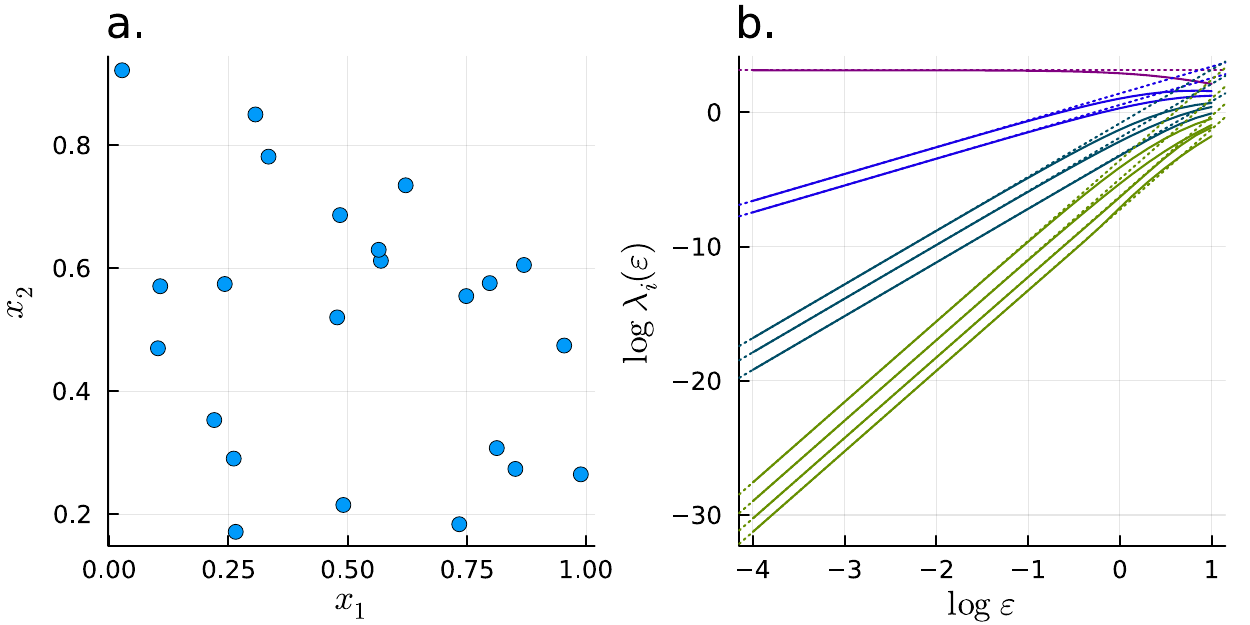}
  \caption{\label{fig:eigen-generic} Eigenvalues of a Gaussian kernel matrix for a unisolvent node set. \textbf{a.} Nodes drawn i.i.d. from the unit square \textbf{b.} The $10$ largest eigenvalues of $\bKe$ for this node set, as a function of $\varepsilon$ (in log-log scale). Groups of eigenvalues with different valuations appear in different colors. There is $1$ eigenvalue with valuation $0$, $2$ eigenvalues with valuation $2$, $3$ eigenvalues with valuation $4$, $4$ eigenvalues with valuation $6$, etc. The asymptotic approximation of an eigenvalue, keeping only the leading term, corresponds to a line in log-log space. These approximations are shown here as dotted lines (note the varying slopes, correspond to different valuations). }
\end{figure}

\begin{figure}[ht]
  \centering
 
  \includegraphics[width=12cm]{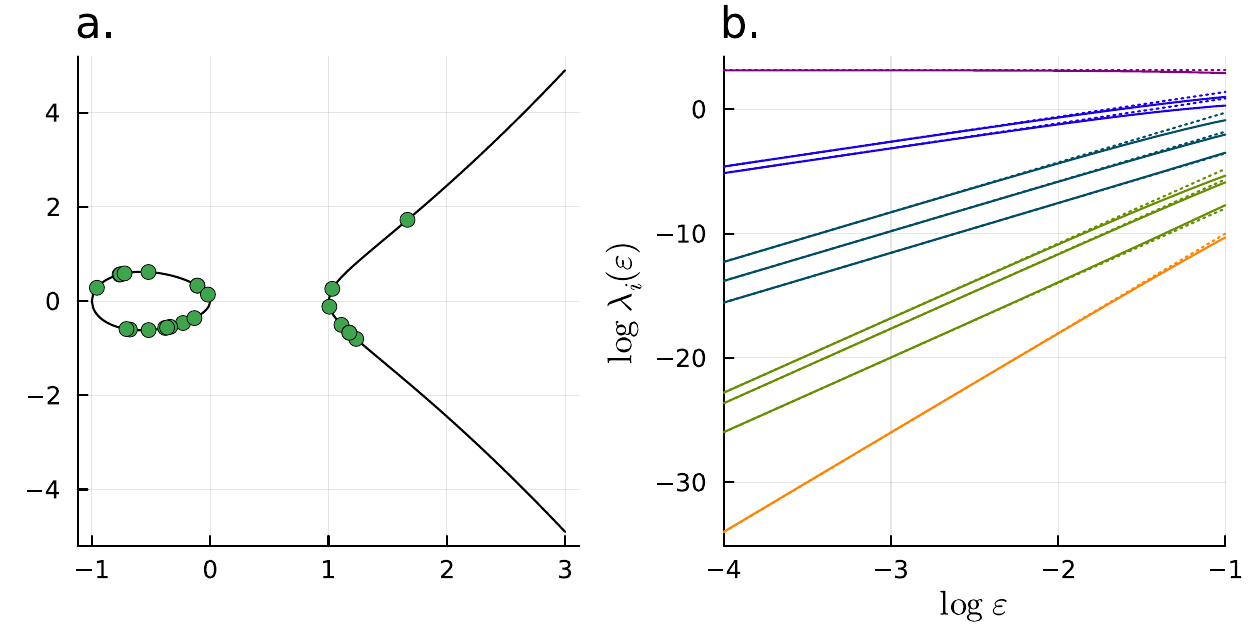}
  \caption{\label{fig:eigen-non-uni} Eigenvalues of a Gaussian kernel matrix for a non-unisolvent node set. \textbf{a.} Nodes drawn i.i.d. from the affine variety $x_2^2 = x_1^3 - x_1$ (the underlying curve is shown as a solid line). This node set is non-unisolvent at degree $3$ and higher, see text.  \textbf{b.} The $10$ largest eigenvalues of $\bKe$ for this node set, as a function of $\varepsilon$ (in log-log scale). Groups of eigenvalues with different valuations appear in different colors. There is $1$ eigenvalue with valuation $0$, $2$ eigenvalues with valuation $2$, $3$ eigenvalues with valuation $4$, but only $3$ eigenvalues with valuation $6$. All further groups of eigenvalues are also of size $\leq 3$. Compared to the unisolvent set, the eigenvalue groups are smaller starting at valuation $8$, which corresponds to polynomials of degree 4.  }
\end{figure}

In our first example (fig. \ref{fig:eigen-generic}), nodes are drawn i.i.d.
uniformly from the unit square. In such configurations, the node set is
unisolvent almost surely. The Gaussian kernel $\psi(r)=\exp(-r^2)$ is completely smooth, and cor. \ref{cor:ase-smooth-unisolvent} applies. As a consequence, the successive blocks of eigenvalues of with asymptotic behaviour in $\varepsilon^0,\varepsilon^2,\varepsilon^4,\dots $ have size equal to $\HH_{0,2},\HH_{1,2},\HH_{2,d},\HH_{3,d},\HH_{4,d},\dots =1,2,3,4,\dots $. In general the $i$-th block has valuation $2i$ and is of size $\HH_{i,d}=i+1$. The leading coefficients can be obtained from eq. \eqref{eq:ase-term-unisolvent}. In fig. \ref{fig:eigen-generic}, we show the eigenvalues of the example matrix as a function of $\varepsilon$, along with the asymptotic approximations.

In our second example (fig. \ref{fig:eigen-non-uni}), the nodes are sampled i.i.d. from the algebraic curve $x_2^2 = x_1^3 - x_1$. This is an affine variety of dimension one and degree 3, and the node set is non-unisolvent for polynomials of degree 3 and higher. More precisely, the blocks $\bR_{i,i}$ in the QR decomposition of the Vandermonde matrix (which have size $\HH_{i,2}=i-1, i>0$) have rank $\leq 3$ for $i\geq 3$. Compared to the unisolvent case, we still have blocks of eigenvalues with asymptotic behaviour in $\varepsilon^0,\varepsilon^2,\varepsilon^4,\dots $ but starting at valuation $8$ (which corresponds to degree $4$) these blocks are all of size $\leq 3$. The leading coefficients of the eigenvalues can be obtained from cor. \ref{cor:ase-non-unisolvent}.

\begin{remark}
Just before the submission of this manuscript, we became aware of the paper by Diab and Batenkov \cite{diab2024spectralpropertiesinfinitelysmooth} that  investigated the asymptotic eigenvalues in the non-unisolvent case, using tools similar to the ones  in   \cite{BarthelmeUsevich:KernelsFlatLimit}.
\end{remark}

\section{ASE in the degenerate case}
\label{sec:iterative-alg}

The goal of this section is to address the general case of positive definite $\bKe$.
In fact, not every matrix $\bKe$ can be reduced to the generalized kernel form \eqref{eq:generalised-kernel}.
As an example, consider the following matrix:
\[
\bKe = 
\begin{pmatrix}
1 & \ep1 & \ep1 \\
\ep1 & 2\ep2 & \ep2 \\
\ep1 & \ep2 & \ep2 + \ep3
\end{pmatrix}.
\]
Indeed, we can represent the matrix in the diagonal scaling form as
\begin{equation}\label{eq:degenerate-example}
\bKe = 
\begin{pmatrix}
1 &  & \\
 & \ep1 & \\
 & & \ep1
\end{pmatrix}
\begin{pmatrix}
1 & 1 & 1 \\
1 & 2 & 1 \\
1 & 1 & 1 + \ep1
\end{pmatrix}
\begin{pmatrix}
1 &  & \\
 & \ep1 & \\
 & & \ep1
\end{pmatrix}.
\end{equation}
We see that the last Schur complement of the last block is
\[
\begin{pmatrix}
2 & 1 \\
1 & 1 
\end{pmatrix}-
\begin{pmatrix}
1\\1
\end{pmatrix}
\begin{pmatrix}
1&1
\end{pmatrix} =\begin{pmatrix}
1 & 0 \\
0 & 0 
\end{pmatrix},
\]
therefore, by  \Cref{thm:ASE-H} the first two eigenvalues are $O(1)$, $O(\ep2)$ and the  last eigenvalue is at least $O(\ep3)$.
In addition, the first two limiting eigenvectors are 
\[
\begin{pmatrix}1 \\ 0 \\ 0 \end{pmatrix}\quad\text{and}\quad \begin{pmatrix}0 \\ 1 \\ 0 \end{pmatrix},
\]
respectively, and therefore the limiting eigenvectors are given by  $\matr{U}_0 = \matr{I}$.

However, we cannot bring the matrix to the form \eqref{eq:generalised-kernel} with the diagonal scaling
\[
\begin{pmatrix}
1 &  & \\
 & \ep1 & \\
 & & \ep{\frac32}
\end{pmatrix},
\]
thus we cannot  use \Cref{thm:ASE-H} to get all the information on the ASE.
In order to deal with such cases, we are going to propose an iterative reduction strategy, also based on Schur complements.

\subsection{Schur complement in the diagonally scaled form}
Then we are able to derive the following result, that helps us to continue the reduction beyond the case in \Cref{thm:ASE-H}.
\begin{theorem}
  \label{thm:iteration}
Let $\bKe$ be partitioned as 
\[
\bKe = \bDe \begin{pmatrix}
\bH_{11}(\varepsilon) & \bH_{12}(\varepsilon) \\ 
\bH_{21}(\varepsilon) & \bH_{22}(\varepsilon) \\ 
\end{pmatrix}
 \bDe,
\]
where $\bH_{11}(\varepsilon) \, \bH_{22}(\varepsilon)$ are $m\times m$ and $(n-m) \times (n-m)$ symmetric, $\bH_{21}(\varepsilon)  = \bH_{21}^{\T}(\varepsilon)$, $\bH_{11}(0)$ invertible  and 
\begin{equation}\label{eq:bDe_degenerate}
\bDe=
\begin{pmatrix}
\boldsymbol{\Delta}_m(\varepsilon)  & \\
& \varepsilon^{\frac{s}{2}} \bI_{n-m}\\
\end{pmatrix},\quad 
\boldsymbol{\Delta}_m(\varepsilon) = 
\diag( \varepsilon^{\gamma_1}, \dots, \varepsilon^{{\gamma_m}}).
\end{equation}
where $\gamma_1 \le \cdots \le \gamma_m < \frac{s}{2}$, where $s$ is an integer.
Define the following block diagonal matrix:
\begin{equation}\label{eq:rotated_matrix_Qe}
\bK'(\varepsilon) =
\begin{pmatrix}
\boldsymbol{\Delta}_m(\varepsilon) \bH_{11}(\varepsilon) \boldsymbol{\Delta}_m(\varepsilon) & 0 \\ 
0    & \varepsilon^{2s} (\bH_{22}(\varepsilon) -  \bH_{21}(\varepsilon)\bH^{-1}_{11}(\varepsilon)\bH_{12}(\varepsilon)) \\ 
\end{pmatrix}.
\end{equation}
Then $\bKe \equivinf \bK'(\varepsilon)$.
\end{theorem}
\begin{remark}
  \Cref{thm:iteration} bears a strong resemblance to theorem 1.1 from
  \cite{Carlsson2024}, which deals with perturbed matrices of the
  form $A+E$ where $E$ is any perturbation (not necessarily analytic), and gives
  an estimate with $\O(\norm{E}^3)$ error for the eigenvalues in which a similar
  Schur complement also appears. 
\end{remark}

Before proving  \Cref{thm:iteration}, we show an example of such a reduction.
\begin{example}
We continue the example from the beginning of the \cref{sec:iterative-alg}.
Applying  \Cref{thm:iteration} in \eqref{eq:degenerate-example} with $\Delta(\varepsilon)$ as in \eqref{eq:degenerate-example},
we get that the Schur complement becomes
\[
\begin{pmatrix}
2 & 1 \\
1 & 1+\varepsilon 
\end{pmatrix}-
\begin{pmatrix}
1\\1
\end{pmatrix}
\begin{pmatrix}
1&1
\end{pmatrix} =\begin{pmatrix}
1 & 0 \\
0 & \varepsilon
\end{pmatrix},
\]
hence the ASE of the matrix in \eqref{eq:degenerate-example} is equal to 
the ASE of 
\[
\begin{pmatrix}
1 & & \\
 & \varepsilon^2 & \\
 & & \varepsilon^3 
\end{pmatrix},
\]
which is already in the ASE form.
\end{example}

\begin{remark}
Note that \Cref{thm:iteration} can be used to obtain an ASE of any symmetric analytic matrix in an iterative fashion.
Indeed, take the leading term  of the right lower block in \eqref{eq:rotated_matrix_Qe} 
\[
\bA(\varepsilon) = \varepsilon^{s} (\bH_{22}(\varepsilon) -  \bH_{21}(\varepsilon)\bH^{-1}_{11}(\varepsilon)\bH_{12}(\varepsilon)),
\]
and assume that $\val (\bA(\varepsilon) ) = \gamma \ge s$.
Then the leading term of the matrix $\bA(\varepsilon) / \varepsilon^{\gamma}$ will describe the term of the ASE for the group of eigenvalues of the next valuation $\varepsilon^{\gamma}$.
By choosing an appropriate rotation $\matr{Q}$, this matrix can be brought to 
\[
\bQ \bA(\varepsilon)\bQ^\T = 
\begin{pmatrix}
\O(\varepsilon^{\gamma}) & \O(\varepsilon^{\gamma+1})\\
  \O(\varepsilon^{\gamma+1}) & \O(\varepsilon^{\gamma+1}) \\
\end{pmatrix},
\]
which is in a diagonally-scaled form and thus \Cref{thm:iteration} can
be applied again (combined with \Cref{thm:ASE-H} if necessary).
\end{remark}

\subsection{Proof of \Cref{thm:iteration}}

In order to prove \Cref{thm:iteration}, we will need several lemmas. 
\begin{lemma}\label{lem:diagonalRIgeneralized}
Let ${\bLa}(\varepsilon)$  be an analytic diagonal matrix.
Then we have that for any positive integer $s$ and a  symmetric  $\bA(\varepsilon) = \O(\varepsilon)$, 
\[
\bM_{s,\tau}(\bLa(\varepsilon)+ \tau\varepsilon^s \bA(\varepsilon)) = \bM_{s,\tau}(\bLa(\varepsilon)).
\]
for all $\tau$ except a finite number of values.
\end{lemma}
\begin{remark}
\Cref{lem:diagonalRIgeneralized} slightly generalizes the derivations in the proof of  \Cref{lem:RI}, where we look at diagonal perturbations of a diagonal matrix.
\end{remark}
\begin{proof}[Proof of  \Cref{lem:diagonalRIgeneralized}]
Let ${\Lambda}(\varepsilon)$ be written in the form
\[
{\bLa}(\varepsilon) = 
\bDe (\widetilde{\bLa}_0 + \O(\varepsilon)) \bDe,\quad
\widetilde{\bLa}_0 = 
\begin{pmatrix}
\widetilde{\lambda}_1 &  & \\
&\ddots& \\
&& \widetilde{\lambda}_n
\end{pmatrix},
\]
where $\bDe$ is as in  \eqref{eq:bDe_degenerate} and  $\widetilde{\lambda}_1,\ldots,\widetilde{\lambda}_m  \neq 0$.
Consider general $\tau$ (not equal to any of $\{-\widetilde{\lambda_k}\}_{k={m+1}}^{n}$) and denote the perturbed matrix $\bLa'(\varepsilon) = \bLa(\varepsilon) + \tau \varepsilon^{s} \bA(\varepsilon)$.
Then, it is easy to see that we can put the matrix $\bLa'(\varepsilon)$ in the diagonally scaled form with $\bDe$
\[
\bLa'(\varepsilon)= \bDe (\widetilde{\bLa}_0 + \O(\varepsilon) + \tau \varepsilon^{s} \bDe^{-1} \bA(\varepsilon) \bDe^{-1}) \bDe = \bDe (\widetilde{\bLa}_0 + \O(\varepsilon)) \bDe.
\]
Thus we can apply \Cref{thm:ASE-H} to get that
\[
\truncASE{\bLa'}{s} = \bDe \widetilde{\bLa}_0  \bDe  = \truncASE{\bLa}{s},
\]
which implies $\bM_{s,\tau}(\bLa') = \bM_{s,\tau}(\bLa)$ by \Cref{prop:RI-asymp-equiv}.
\end{proof}

\subsubsection{Main result}
Then Lemma~\ref{lem:diagonalRIgeneralized} implies the following.
\begin{corollary}\label{cor:changing_ASE}
Let $\bQ(\varepsilon) = \bI + O(\varepsilon)$.
Then the congruence with $\bQ(\varepsilon)$ preserves the ASE for any matrix:
\[
\bQ(\varepsilon)\bKe \bQ^{\T}(\varepsilon) \equivinf \bKe.
\]
\end{corollary}
\begin{proof}
Let $\bKe = \bU(\varepsilon) \bLa(\varepsilon) \bU^{\T}(\varepsilon)$ be the analytic EVD of $\bKe$ and denote
\[
\bK'(\varepsilon) \eqdef \bQ(\varepsilon)\bKe \bQ^{\T}(\varepsilon) = \bB(\varepsilon)  \bLa(\varepsilon) \bB^{\T}(\varepsilon),
\]
where $\bB(\varepsilon) \eqdef \bQ(\varepsilon)\bU(\varepsilon)$. Note that $\bB(\varepsilon)  = \Ulim + \O(\varepsilon)$ and hence we have that $\bB^{-1}(\varepsilon)  = \Ulim^{\T} + \O(\varepsilon)$, and
\[
\bI = \bB(\varepsilon) ( \bB^{-1}(\varepsilon) \bB^{-\T}(\varepsilon))  \bB^{\T}(\varepsilon) =  \bB(\varepsilon) ( \bI + \O(\varepsilon))  \bB^{\T}(\varepsilon) 
\]
Using this identity, for any integer $s$, we get that
\begin{align*}
\bM(\tau\varepsilon^{s},\bK') & = \bI - \tau \ep{s}(\bK'(\varepsilon) + \tau \ep{s}\bI)^{-1} 
= \bI -\tau \ep{s}(\bB(\varepsilon) \bLa(\varepsilon) \bB^{\T}(\varepsilon) + \tau \ep{s}\bI)^{-1} \\
&= \bI -\tau \ep{s}(\bB(\varepsilon)(\bLa(\varepsilon) + \tau \ep{s} (\bI + O(\varepsilon))) \bB^{\T}(\varepsilon) )^{-1}  \\
&= \bI -\bB^{-\T}(\varepsilon) \left(\tau \ep{s} \left(\bLa(\varepsilon) + \tau \ep{s} (\bI + O(\varepsilon)) \right)^{-1} \right)\bB^{-1}(\varepsilon)\\
 &= \bI -\left(\bI + \O(\varepsilon) \right) \left(\Ulim + O(\varepsilon) \right) \left(\bI - \bM(\tau\varepsilon^{s},\bLa(\varepsilon) + \O(\varepsilon)) \right) \left(\Ulim^{\T} + O(\varepsilon) \right)  (\bI + \O(\varepsilon))\\
 &= \bI -\left(\bI + \O(\varepsilon) \right) \left(\Ulim + O(\varepsilon) \right) \left(\bI - \bM_{s,\tau}(\bLa) + \O(\varepsilon) \right) \left(\Ulim^{\T} + O(\varepsilon) \right)  \left(\bI + \O(\varepsilon) \right)
\end{align*}
where the equality is due to \Cref{lem:diagonalRIgeneralized}.
By taking the limit in $\varepsilon$, we get $\bM(\tau\varepsilon^{s},\bK') = \bM(\tau\varepsilon^{s},\bK)$, which is valid for all nonnegative integer $s$.
Then  we have  that $\bK' \equivinf \bKe$ by \Cref{prop:RI-asymp-equiv}.
\end{proof}

\begin{proof}[Proof of \Cref{thm:iteration}]
Consider the following analytic matrix
\[
\bQ(\varepsilon) = 
\begin{pmatrix}
\bI_{m} & \zeroes \\
- \varepsilon^{s} \bH_{21}(\varepsilon)\bH^{-1}_{11}(\varepsilon) \boldsymbol{\Delta}^{-1}_m(\varepsilon)   &  \bI_{n-m}
\end{pmatrix}
\]
Then we can verify that 
\begin{align*}
&\bQ(\varepsilon)\bKe\bQ^{\T}(\varepsilon) =  \\
&=\begin{pmatrix}
\bI_{m} & \zeroes \\
- \varepsilon^{s} \bH_{21}(\varepsilon)\bH^{-1}_{11}(\varepsilon) \boldsymbol{\Delta}^{-1}_m(\varepsilon)   &  \bI_{n-m}
\end{pmatrix}  \bDe
\begin{pmatrix}
\bH_{11}(\varepsilon) & \bH_{12}(\varepsilon) \\ 
\bH_{21}(\varepsilon) & \bH_{22}(\varepsilon) \\ 
\end{pmatrix}\bDe\begin{pmatrix}
\bI_{m} & - \varepsilon^{s} \boldsymbol{\Delta}^{-1}_m(\varepsilon)\bH^{-1}_{11}(\varepsilon)\bH_{12}(\varepsilon) \ \\
\zeroes   &  \bI_{n-m}
\end{pmatrix} \\
 &= \begin{pmatrix}
\boldsymbol{\Delta}_m(\varepsilon) \bH_{11}(\varepsilon) \boldsymbol{\Delta}_m(\varepsilon) & 0 \\ 
0    & \varepsilon^{2s} (\bH_{22}(\varepsilon) -  \bH_{21}(\varepsilon)\bH^{-1}_{11}(\varepsilon) \boldsymbol{\Delta}^{-1}_m(\varepsilon)\boldsymbol{\Delta}_m(\varepsilon)\bH_{11}(\varepsilon) \boldsymbol{\Delta}_m(\varepsilon)\boldsymbol{\Delta}^{-1}_m(\varepsilon)\bH_{12}(\varepsilon)),\\ 
\end{pmatrix}  \\
&= \bK'(\varepsilon),
\end{align*}
and the proof is complete by \Cref{cor:changing_ASE}.
\end{proof}

\section{Conclusion}
\label{sec:conclusion}

We hope to have convinced the reader that Theorems \ref{thm:ASE-H},
\ref{thm:ASE-generalised-kernel-form} and \ref{thm:iteration}, can be used to simplify the
analysis of matrix perturbations. One noteworthy limitation is that we have
assumed that the perturbations are analytic, i.e. the classical framework used
by Rellich and Kato. This limitation can be lifted, if one instead looks at the
matrix $\bKe$ as admitting an asymptotic series (which need not be a power
series). We intend to extend our results in this direction in future work. 

\section{Appendix}
\label{sec:appendix}

\begin{proof}[Proof of \Cref{prop:asymptotic-kernel-smooth}]
As shown in the proofs of \cite[Theorems  4.5 and 6.3]{BarthelmeUsevich:KernelsFlatLimit}, under such assumptions, the kernel matrix has expansion
\begin{equation}\label{eq:kernel_radial_expansion}
\bKe  = \bV_{\le p-1} \matr{\Delta}_{p-1} \bW_{\le p-1} \matr{\Delta}_{p-1} \bV^{\T}_{\le p-1} +
\varepsilon^p(\bV_{\le p-1} \matr{\Delta}_{p-1} \bW_{1}(\varepsilon) + 
\bW_{2}(\varepsilon) \matr{\Delta}_{p-1} \bV^{\T}_{\le p-1} ) + \varepsilon^{2p-1} (\matr{W}_3(\varepsilon)),
\end{equation}
where $\matr{W}_3(\varepsilon)=\varepsilon^{2p-1} (\psi_{2p-1} \bD^{(2p-1)} + \O(\varepsilon))$ and $\psi_{2p-1} = 0$ if $p < r$.
Note that since $\bV = \bV_{\le p-1}$ is full row rank, 
we have 
\begin{equation}\label{eq:identity_decomposition_smooth}
\bI_n = \bV \matr{\Delta}_{p-1} \matr{\Delta}_{p-1}^{-1} \bV^{\dagger},
\end{equation}
hence, we can rewrite
\[
\varepsilon^p(\bV_{\le p-1} \matr{\Delta}_{p-1} \bW_{1}(\varepsilon) + 
\bW_{2}(\varepsilon) \matr{\Delta}_{p-1} \bV^{\T}_{\le p-1} ) = 
\bV \matr{\Delta}_{p-1}   (\widetilde{\bW}_{1}(\varepsilon)  + \widetilde{\bW}_{2}(\varepsilon)  + )\matr{\Delta}_{p-1}   \bV^\T,
\]
where 
\[
\widetilde{\bW}_{2}(\varepsilon) =  \varepsilon^{p} \bDe^{-1} \bV^{\dagger} \bW_2(\varepsilon) = \O(\varepsilon),\quad \widetilde{\bW}_{1}(\varepsilon) =  \varepsilon^{p} \bW_1(\varepsilon)  (\bV^{\dagger})^{\T} \bDe^{-1} = \O(\varepsilon).
\]
Similarly, for $\bW_3(3)$, we have
\[
\varepsilon^{2p-1} \bW_3(3) =  \bV \matr{\Delta}_{p-1} \underbrace{\varepsilon^{p-1} \matr{\Delta}_{p-1}^{-1} \bV^{\dagger}
(\varepsilon \bW_3(\varepsilon)) (\bV^{\dagger})^\T \varepsilon^{p-1} \matr{\Delta}_{p-1}^{-1}}_{\widetilde{\matr{W}}_{3}(\varepsilon) = \O(\varepsilon)}\matr{\Delta}_{p-1} \bV.
\]
Combining it all together, we obtain
\[
\bKe = \bV \bDe (\bW + \underbrace{\widetilde{\bW}_{1}(\varepsilon) + \widetilde{\bW}_{2}(\varepsilon) + \widetilde{\bW}_{3}(\varepsilon)}_{\O(\varepsilon)})  \bDe  \bV^\T,
\]
which completes the proof.
\end{proof}

\begin{proof}[Proof of \Cref{prop:asymptotic-kernel-smooth-general}]
The proof repeats that of \Cref{prop:asymptotic-kernel-smooth}, but instead of \eqref{eq:kernel_radial_expansion} we use another expansion from  \cite[eqn. (32),(55)]{BarthelmeUsevich:KernelsFlatLimit}, which reads
\[
\bKe = \bV \bDe \bW \bDe  \bV^\T +  \varepsilon^{p} \bV \bDe \bW_{1}(\varepsilon) + \varepsilon^{p}\bW_{2}(\varepsilon) \bDe  \bV^\T  + \varepsilon^{2p}\bW_{3}(\varepsilon).
\]
\end{proof}

\begin{proof}[Proof of \Cref{prop:asymptotic-kernel-finite-smoothness}]
We use the expansion \eqref{eq:kernel_radial_expansion} for $p = r$, and an idea similar to the one in \eqref{eq:identity_decomposition_smooth}, but for $\bV = \left[ \bV_{\leq r-1}\ \bA \right]$.
Define the matrix $\widetilde{\bV}$ as 
\[
\widetilde{\bV} = \begin{bmatrix} \bV_{\leq r-1}^{\dagger} \\ \bA^{\dagger}\end{bmatrix}.
\]
Then, since $\bV$ is full row rank and the matrices $\bV_{\leq r-1}$ and  $\bA$ span orthogonal subspaces, we have
\begin{equation*}
\bI_n = \bV \widetilde{\bV} = \bV \bDe \bDe^{-1} \widetilde{\bV},
\end{equation*}
hence, we can rewrite
\begin{align*}
\varepsilon^r \bV_{\le r-1} \matr{\Delta}_{r-1} \bW_{1}(\varepsilon)  & = 
\bV_{\le r-1} \matr{\Delta}_{r-1} \widetilde{\bW}_{1}(\varepsilon) \bDe  \bV^\T,\\
\varepsilon^r  \bW_{2}(\varepsilon) \matr{\Delta}_{r-1} \bV^{\T}_{\le r-1}  & =
\bV \bDe \widetilde{\bW}_{2}(\varepsilon)  \matr{\Delta}_{r-1} \bV^\T_{\le r-1},
\end{align*}
where
\begin{align*}
\widetilde{\bW}_{1}(\varepsilon) =  \varepsilon^{r} \bW_1(\varepsilon)  (\widetilde{\bV})^{\T} \bDe^{-1} = o(1).
\widetilde{\bW}_{2}(\varepsilon) =  \varepsilon^{r} \bDe^{-1} \widetilde{\bV} \bW_2(\varepsilon) = o(1).
\end{align*}
Similarly, for $\bW_3(3)$, we have
\[
\varepsilon^{2r-1} \bW_3(3) =  \bV (\varepsilon^{r-\frac{1}{2}}\matr{I}) \widetilde{\bV} \bW_3(\varepsilon)  (\widetilde{\bV})^\T (\varepsilon^{r-\frac{1}{2}}\matr{I})  \bV^\T.
\]
Note that
\[
\widetilde{\bV} \bW_3(\varepsilon)  (\widetilde{\bV})^\T = 
\begin{bmatrix}
\bV_{\le r-1}^{\dag}  \bD^{(2r-1)} (\bV_{\le r-1}^{\dag})^{\T} & \bV_{\le r-1}^{\dag}   \bD^{(2r-1)} (\bA^{\dag})^{\T}\\
\bA^{\dag}  \bD^{(2r-1)} (\bV_{\le r-1}^{\dag})^{\T} & \bA^{\dag}  \bD^{(2r-1)} (\bA^{\dag})^{\T}
\end{bmatrix}
+ \O(\varepsilon),
\]
and hence
\[
(\varepsilon^{r-\frac{1}{2}}\matr{I}) \widetilde{\bV} \bW_3(\varepsilon)  (\widetilde{\bV})^\T (\varepsilon^{r-\frac{1}{2}}\matr{I}) =
\bDe \left( \begin{bmatrix}
0 & 0 \\
0 & \bA^{\dag}  \bD^{(2r-1)} (\bA^{\dag})^{\T}
\end{bmatrix} + o(1) \right) \bDe
\]
Combining it all together, we obtain
\[
\bKe = \bV \bDe (\left( \begin{bmatrix}
\bW_{\le r-1} & 0 \\
0 & 0
\end{bmatrix} +  \begin{bmatrix}
0 & 0 \\
0 & \bA^{\dag}  \bD^{(2r-1)} (\bA^{\dag})^{\T}
\end{bmatrix} + o(1) \right) )  \bDe  \bV^\T,
\]
which completes the proof.
\end{proof}

\bibliographystyle{plainnat}

\bibliography{../flat_limit.bib}

\end{document}